\documentclass[a4paper,reqno]{amsart}

\usepackage[a4paper,width={16cm},left=2.5cm,bottom=3cm, top=3cm]{geometry}
\usepackage{enumerate}
\usepackage{yfonts}

\usepackage[english]{babel}

\usepackage{a4wide}
\usepackage[dvipsnames]{xcolor}
\usepackage{amsmath}
\usepackage{amssymb}
\usepackage{amsfonts}
\usepackage{amsthm,graphicx}
\usepackage{comment}
\usepackage{csquotes}
\usepackage{esint}
\usepackage{hyperref}
\hypersetup{
	linktocpage=true,
	colorlinks=true,
	linkcolor=blue!70!black,
	citecolor=orange!40!red,
	urlcolor=magenta!80!black,
}

\usepackage{mathtools}
\mathtoolsset{showonlyrefs}

\numberwithin{equation}{section}
\newtheorem{theorem}{Theorem}[section]
\newtheorem{lemma}[theorem]{Lemma}
\newtheorem{corollary}[theorem]{Corollary}

\newtheorem{proposition}[theorem]{Proposition}
\theoremstyle{definition}  
\newtheorem{definition}[theorem]{Definition}

\newtheorem{remark}[theorem]{Remark}

\newcommand{\mc}{\mathcal}
\newcommand{\mb}{\mathbb}

\newcommand{\la}{\lambda}
\newcommand{\norm}[1]{\left\lVert#1\right\rVert}
\newcommand{\pd}[2]{\frac{\partial#1}{\partial#2}}

\newcommand{\R}{\mb{R}}

\newcommand{\N}{\mb{N}}

\newcommand{\e}{\varepsilon}

\newcommand{\dive}{\mathop{\rm{div}}}

\newcommand{\nota}[1]{\marginpar{\color{magenta}\tiny #1}}

\usepackage{subfig}
\usepackage{tikz}
\usepackage{xcolor}
\usepackage{pgfplots}
\pgfplotsset{compat=1.18}
\usepgfplotslibrary{fillbetween}
\usepackage{mathrsfs}
\usetikzlibrary{arrows}
\usetikzlibrary{shadings}
\makeatletter

\date{\today}
\begin{document}

\title[On a multiphase vectorial Bernoulli free boundary problem]{On a multiphase vectorial Bernoulli free boundary problem}
\author[G. Siclari and B. Velichkov]{Giovanni Siclari and Bozhidar Velichkov}

\address{Bozhidar Velichkov
\newline \indent Dipartimento di Matematica
\newline \indent Universita di Pisa 
\newline\indent Largo Bruno Pontecorvo, 5, 56127 Pisa, Italy}
\email{bozhidar.velichkov@unipi.it}

\address{Giovanni Siclari 
\newline \indent Centro di Ricerca Matematica Ennio De Giorgi
\newline \indent Scuola Normale Superiore di Pisa
\newline\indent Piazza dei Cavalieri 3, 56126 Pisa, Italy}
\email{giovanni.siclari@sns.it}

\begin{abstract}
We study the regularity of minimizers of  a multiphase vectorial Bernoulli free boundary problem. This problem consists in  a minimization problem for the Bernoulli functional  over  families of Sobolev functions with disjoint supports and non trivial grouping. We prove that minimizers exist, are locally Lipschitz continuous, and that their free boundaries do not contain points where three or more phases meet. Our main regularity result establishes that the free boundary is locally a $C^{1,\eta}$ graph near two-phase and branching points for some $\eta >0$.
\end{abstract}

\maketitle


\medskip 


\section{Introduction}\label{sec_intro}
Let  $D$ be  an open bounded set in $\R^d$. In this paper, we are interested in the regularity of minimizers of a multiphase vectorial Bernoulli free boundary problem in 
$D$. Roughly speaking, the terms "multiphase," "vectorial," and "Bernoulli" describe a partition problem for the Dirichlet energy involving a measure term and  non-trivial grouping, that is,  we consider  families of Sobolev functions possessing mutually disjoint supports across different families. However, the supports can have non empty intersection within a given family.

\subsection{Setting of the problem}\label{sec_setting} Let $k\in\N$ and $n\ge k$. We fix a partition of $\{1,\dots,k\}$ into $n$ disjoint sets $A_1,\dots,A_n$, precisely:
\begin{equation}
\bigcup_{j=1}^n A_j= \{1,\dots,k\}, \quad A_j \cap A_{j'} =\emptyset \text{ if }j \neq j' 
\quad \text{and} \quad  \max A_j= \min A_{j+1}-1.
\end{equation}
We set $k_j$ to be the number of elements of $A_j$, $k_j:=|A_j|$, and, for each $i\in\{1,\dots,k\}$, 
we denote by $j_i$ the unique $j\in\{1,\dots,n\}$ such that $i \in A_j$.

For any  $W \in H^1(D, \R^k)$, we define 
\begin{equation}
W_j:D \to \R^k, \quad (W_j)_i:=
\begin{cases}
w_i, &\text{ if } i \in A_j,\\
0, &\text{ if } i \not\in A_j,
\end{cases}
\end{equation}
which we identify with $W_j:D \to \R^{k_j}$ (ignoring the zero components), and 
\begin{equation}
\Omega_W:= \bigcup_{i=1}^k\Omega_{w_i}, \quad \Omega_{W_j}:=\bigcup_{i\in A_j}\Omega_{w_i} \quad \text{ with } \Omega_{w_i}:=\{x \in D: w_i \neq 0\}.
\end{equation}
Let us consider the functional
\begin{equation}\label{def_J}
J_{\Lambda}(W,D):=\int_D |\nabla W|^2\, dx + \sum_{j=1}^n \Lambda_j\left|\Omega_{W_j}\right|,
\end{equation}
with $\Lambda_j >0$ for $j=1, \cdots, n$, over 
\begin{equation}\label{def_H1_Sigma}
H^1(D,\Sigma):=\{W \in H^1(D,\R^k): w_i w_{i'}=0 \text{ if } i \in A_j \text{ and } i' \in A_{j'} \text{ with } j \neq j'\},
\end{equation}
that is, the elements of $H^1(D, \R^k)$ which take values in 
\begin{equation}\label{def_Sigma}
\Sigma:=\bigcup_{j=1}^n R_j, \quad  \text{ where } R_j:=\{x \in \R^k: x_i=0 \text{ if } i \not\in A_j\}.
\end{equation}
Given $g \in H^1(D, \Sigma)$ we are interested in the regularity of minimizers of the problem 
\begin{equation}\label{prob_min}
\inf\{J_{\Lambda}(W,D): W \in H^1(D, \Sigma),W-g \in H_0^1(D, \Sigma)\}
\end{equation}
and of their free boundaries.

\subsection{A short review of the literature}\label{sec_lit}
Problem \ref{prob_min} combines two well-studied classes of free boundary problems:  partition problems and  vectorial Bernoulli free boundary problems. Indeed, similar to a partition problem, we have $n$ disjoint phases but  our functional consist of the Dirichlet energy with a measure term. In recent years, both classes of free boundary problems have been subject of intense investigation by the mathematical community. 

On the one hand, concerning partition problems for the Dirichlet energy without measure terms, basic properties as Lipschitzianity can be establish with very different theories. We mention \cite{GW_harmonic,KS_harmonic} for harmonic maps and  \cite{CTV_nehari,CTV_asym,CTV_var} for limit of competition-diffusion models. Furthermore, there are strong connections with (spectral) shape optimization problems, see \cite{AMET_spect,CFL_sing,CF_opt,CTV_non_line,CTV_opt_part} and  the theory of  symmetric $2$-valued harmonic functions as pointed out for example in \cite{OV_survey}. Regularity of the free boundary had been  investigated in the literature by means of the Almgren monotonicity formula as in \cite{GW_harmonic,TT_reg_seg}, an approach made possible by the absence of measure terms, which allow to classify points in the free boundary according to their Almgren frequency. We refer for example to \cite{HHOT_dim2,OV_triple,ST_liov}. 
Finally, see  the detailed recent survey  \cite{OV_survey}  for further information on the subject.

On the other hand,  for the vectorial Bernoulli free boundary problem  basic properties as Lipschitzianity and non-degeneracy of  minimizers, interior density estimates and local finiteness of the perimeter of the free boundary 
have been investigated  in \cite{BMV_Lip,CSY_vectorial,KL_non_deg,KL_deg,MTV_epsilon,MTV_reg_spec, MTV_reg_vect,T_lip,T_Reg}.
Finer  results, as the regularity of the free boundary at one phase points, has been established with different techniques in several papers see \cite{CSY_vectorial,DT_impr,KL_non_deg,KL_deg,MTV_harn,MTV_reg_spec,MTV_reg_vect,SV_epi}. Singular cones has also been studied, see \cite{CJK_cones_stab,EE_quant_strat,JS_cones,W_dim_red} while the regularity at point of density one of the free boundary is still subject of current investigation, see \cite{MTV_reg_vect,SV_blow_ups} for a characterization of the blow-ups limits at such points. Finally, we mention the comprehensive survey \cite{TV_survey} for more details.

 Furthermore, several shape optimization results are related to the problem we are dealing with. We refer in particular to the  minimization of  functionals with  eigenvalues instead of the Dirichlet energy and a measure term. More precisely, let $\Omega$ be open and $\la_i(\Omega)$ the number $i$ Dirichlet eigenvalues of $\Omega$ for the Laplacian.  Regularity of minimizers of the functional $\la_1(\Omega)+\Lambda|\Omega|$  was investigated in \cite{BL_spec_meaus_1}, of $\la_2(\Omega)+\Lambda|\Omega|$  in \cite{KL_deg,MTV_spect_2} and of $\sum_{i=1}^n\la_i(\Omega)+\Lambda|\Omega|$ in \cite{MTV_reg_spec}, see in particular \cite[Lemma 2.4]{MTV_reg_spec}.
We also mention that as corollary of the analysis of two-phase and branching points for the two phase Bernoulli functional in  \cite{DPSV_branch}, a free boundary regularity result was obtained for the functional   $\sum_{i_=1}^k(\la_1(\Omega_i)+\Lambda_i|\Omega_i|)$, where $\Omega_i \cap \Omega_{U_j}=\emptyset$ for $i \neq j$. This shape optimization problem was introduced in \cite{BV_multi} and can be seen, in particular, as a scalar version of our problem for eigenvalues, that is, with trivial grouping, all the support are disjoint.

\subsection{Our main results}
Roughly speaking, in this paper, we show that any minimizer $U$  of \eqref{prob_min} is locally Lipschitz continuous in $D$ and that the free boundary $\partial\Omega_U$ is the union of two $C^1$ graphs in a neighborhood of two-phase and branching points whose definition we now recall. 

\begin{definition}\label{def_two_phase_branching}
We say that $x \in  D$ is a \textit{two-phase point} if $x \in \partial\Omega_{U_j}\cap \partial\Omega_{U_{j'}}$ for some $j \neq j'$ and there exists $r>0$ such that the $B_r(x) \cap \{|U|=0\}=\emptyset$. We say that  $x \in  D$ is a \textit{branching point} if $x \in \partial\Omega_{U_j}\cap \partial\Omega_{U_{j'}}$   for some $j \neq j'$ and   $B_r(x) \cap \{|U|=0\}\neq\emptyset$ for any $r>0$.
\end{definition}

The following theorem is our main result.
\begin{theorem}\label{theo_reg_diri}
For all $g \in H^1(D, \Sigma)$ there exists a minimizer of \eqref{prob_min}. Furthermore, any minimizer $U$   is locally Lipschitz continuous in $D$ and 
$\partial\Omega_{U_j}\cap \partial\Omega_{U_{j'}} \cap \partial\Omega_{U_{j''}}\cap D=\emptyset$ for all distinct  $j,j',j'' \in \{1,\dots,n\}$.
Finally, if $x \in   \partial \Omega_{U_j}\cap \partial \Omega_{U_{j'}} \cap D$ 
for some $j\neq j' \in \{1,\dots,n\}$, then there exists $r>0$ such that $B_r(x)\cap \partial \Omega_{U_j}, B_r(x)\cap \partial \Omega_{U_{j'}}$ are $C^{1,\eta}$ graphs for some $\eta>0$.
\end{theorem}

\begin{remark}
In  Theorem \ref{theo_reg_diri} we do not consider \textit{one phase points}, that is, $x \in D$
such that $x \in \partial \Omega_{U_j}$ for a unique $j \in \{1,\dots,n\}$. The reason why is that in a neighborhood  $B_r(x)$ of such points $U_j$ is a minimizer of the  vectorial Bernoulli free boundary functional, that is, the functional
\begin{equation}\label{J_vec_ben}
J_{\Lambda_j}(W,B_r(x)):=\int_{B_r(x)} |\nabla W|^2\, dx +  \Lambda_j\left|\Omega_{W}\right|.
\end{equation}
The regularity of  minimizers of \eqref{J_vec_ben} has been investigated in several papers, see Section \ref{sec_lit}.
\end{remark}

The proof of Theorem \ref{theo_reg_diri} can be summarized as follows. It is not hard to establish  existence with variational methods. Non-degeneracy  is easy to obtain as well since we only need to use inward variations on each phase and so we can proceed as in \cite{AC_one_phase,MTV_reg_spec}. On the other hand, proving  Lipschitzianity is far more challenging. Indeed, the presence of a measure term prevents us from using the strategy usually employed in partition problems briefly discussed in  Section \ref{sec_lit}. On the other hand the fact that we are dealing with more than one phase makes harder to use minimality  directly by restricting the class of possible comparators. 
Hence, letting $U$ be a minimizer  of \eqref{prob_min}, we use a delicate iterative argument based on the Alt-Caffarelli-Friedman's monotonicity formula to prove the boundedness in $r$ of $\frac{1}{r}\fint_{\partial B_r(x)} |U_j|\, d \mc{H}^{d-1}$ for $x \in \partial\Omega_U$ and any $j \in\{1,\dots,n\}$. Such a bound implies Lipschitzianity is a standard way. The same argument also allows us to exclude the presence of triple points.\medskip

Once, basic properties have been established, we turn to the regularity of the free boundary. The main idea is to prove that in a neighborhood of two-phase or branching points the domains $\Omega_{U_j},  \Omega_{U_{j'}}$ are $(\delta,R)$-Reifenberg flat, see Definition \ref{def_Reif_flat}. This property allows us to reduce the problem to a scalar one (in the spirit of \cite{MTV_reg_spec,MTV_reg_vect}) and then invoke the regularity results proved in \cite{DPSV_branch} thus yielding Theorem \ref{theo_reg_diri}. To prove that the free boundaries $\partial\Omega_{U_j}$ and $\partial \Omega_{U_{j'}}$ are  Reifenberg-flat in a neighborhood of a contact (two-phase or branching) point, we will use an argument based on the Weiss monotonicity formula, inspired by \cite{MTV_reg_spec}. We notice that, contrary to \cite{MTV_reg_spec}, in a neighborhood of a branching point we can have free boundary points of all Lebesgue densities between $1/2$ and $1$, which is due to the fact that we cannot a priori exclude the presence of one-phase singular cones approaching a branching point. Thus, we first show that the blow-ups with variable centers at two-phase or branching points are flat, and we use the improvement of flatness for minimizers of the vectorial Bernoulli free boundary problem (see \cite{DT_impr}) to transfer this information to the one-phase points near a branching point. 


\medskip
The paper is organized as follow. In Section \ref{sec_min_basic} we establish basic properties of minimizers as existence, non-degeneracy, Lipschitzianity, density estimates and absence of triple points.   In Section \ref{sec_blowup_Weiss} we study blows-up with variable center at two-phase points by the means of Weiss monotonicity formula. Finally, in Section  \ref{sec_reg}, we prove the Reifenberg flatness of the free boundary and conclude the proof of Theorem \ref{theo_reg_diri}.

\medskip

\textbf{Notation:} Throughout the paper, we are always going to denote minimizers of \eqref{prob_min} with $U$ and comparators with either $W$ or $V$.

\section{Existence of minimizers, non degeneracy and Lipschitzianity}\label{sec_min_basic}
In this section we establish basic properties of minimizers. We obtain the existence and the non-degeneracy of the minimizers in Proposition \ref{prop_existence_min} and Proposition \ref{prop_non_dege}. The key result of the section is Proposition \ref{prop_multi_phases_means}, which allows to prove both the Lipschitzianity of the minimizers (Proposition \ref{prop_lip}) and the absence of triple points (Corollary \ref{corol_no_triple}). In this proposition we prove, in particular, the boundedness of the function
$$r\mapsto \left(\frac{1}{r}\fint_{\partial B_r(x)} |U_j|\, d \mc{H}^{d-1}\right)\left(\frac{1}{r}\fint_{\partial B_r(x)} |U_{j'}|\, d \mc{H}^{d-1}\right),$$
by the Alt-Caffarelli-Friedman's monotonicity formula (Theorem \ref{theo_multi_mono}), the subharmonicity of $|U_j|$ (Proposition \ref{prop_existence_min}) and the non-degeneracy estimate (Proposition \ref{prop_non_dege}), combined with an iteration argument on dyadic scales.  

\subsection{Existence}
We begin with existence.

\begin{proposition}\label{prop_existence_min}
For any  $g \in H^1(D, \Sigma)$  there exists a minimizer $U \in  H^1(D, \Sigma)$ of \eqref{prob_min}. Furthermore
\begin{enumerate}[(i)]
    \item $u_i$  is harmonic on  $\Omega_{U_j}$  for any $i \in A_j$, and any $j=1,\dots,n$;
    \item $|U|$ and $|U_j|$  are subharmonic on $D$ in a distributional sense for any  $j= 1, \dots,n$.
\end{enumerate}
\end{proposition}

\begin{proof}
Let $\{U_h\} \subset H^1(D, \Sigma)$ be a minimizing sequence. By minimality, 
\begin{equation}
\int_D |\nabla U_h|^2 \, dx \le J_{\Lambda}(U_h,D) \le  J_{\Lambda}(g,D).
\end{equation}
Hence, by the Poincaré inequality, for some positive constant $C$ depending only on $D$
\begin{equation}
\int_D |U_h-g|^2\, dx \le C \int_D |\nabla(U_h-g)|^2\, dx \le 2CJ_{\Lambda}(g,D).
\end{equation}
We conclude that $\{U_h\}$ is bounded $H^1(D, \R^k)$ thus, up to a subsequence, $U_h \rightharpoonup U$ weakly in 
$H^1(D, \R^k)$ and pointwise a.e in $D$ to some $U \in H^1(D, \R^k)$.

For a.e. $x \in D$ and any $h \in \mathbb{N}$, $ (u_h)_i (u_h)_{i'}=0 \text{ if } i \in A_j \text{ and } i' \in A_{j'} \text{ with } j \neq j'$ so that, passing to the limit as $h \to \infty$, we conclude  that $U \in H^1(D,\Sigma)$.
Furthermore, for a.e. $x \in D$ and any $i=1,\dots,k$
\begin{equation}
\chi_{\Omega_{u_i}}(x) \le \liminf_{k \to \infty}\chi_{\Omega_{(u_h)_i}}(x)  
\end{equation}
thus, by Fatou's Lemma,  integrating over $D$ we obtain for any $j=1,\dots,n$
\begin{equation}
\left|\Omega_{U_j}\right| \le \liminf_{h \to \infty} \left|\Omega_{(U_j)_h}\right|.
\end{equation}
Hence, by the lower semicontinuity of norms and the strong convergence in $L^2(D, \R^n)$,
\begin{equation}
J_{\Lambda}(U,D) \le \liminf_{h \to \infty}J_{\Lambda}(U_h,D),
\end{equation}
that is, $U$ is a minimizer of \eqref{prob_min}.

For any $i=1, \dots, k$ and any $w \in H^1(D)$ with $w-u_i \in H_0^1(\Omega_{U_j})$  and $i \in A_j$,   let us consider the comparator $w_i:=(u_1,\dots,u_{i-1}, w,u_{i+1}\dots, u_k)$.
Testing \eqref{prob_min} with $w_i$ we get
\begin{equation}
\int_D |\nabla u_i|^2 \, dx +\Lambda_j |\Omega_{U_j}| \le \int_D |\nabla w|^2 \, dx +\Lambda_j |\Omega_{W_j}|.
\end{equation}
Since $\Omega_{W_j} \subset \Omega_{U_j}$, we must have
\begin{equation}
\int_D |\nabla u_i|^2 \, dx \le \int_D |\nabla w|^2 \quad \text{ for any } w \in  H^1(D) \text{ with } w-u_i \in H_0^1(\Omega_{U_j}),
\end{equation}
that is, $u_i$ is harmonic on $\Omega_{U_j}$.
Furthermore,  on $\Omega_U$,
\begin{multline}
\Delta |U|= \dive(\nabla |U|)=\sum_{i=1}^k \Delta u_i \frac{u_i}{|U|}  +\sum_{i=1}^k\nabla \left(\frac{u_i}{|U|}\right) \cdot \nabla u_i\\
= \frac{1}{|U|^3} \left[|U|^2\sum_{i=1}^k|\nabla u_i|^2 -\sum_{i,i'=1}^ku_iu_{i'} \nabla u_i \nabla u_{i'} \right]\\
=\frac{1}{|U|^3} \left[\sum_{i'=1}^n\sum_{i=1}^k|u_{i'}|^2|\nabla u_i|^2 
-\sum_{i,i'=1}^ku_iu_{i'} \nabla u_i \nabla u_{i'} \right]\ge 0,
\end{multline}
since  for any  $i,i'=1, \dots k$
\begin{equation}
|u_{i'}|^2|\nabla u_i|^2+|u_i|^2|\nabla u_{i'}|^2-2u_iu_{i'} \nabla u_i \nabla u_{i'} \ge 0.
\end{equation}
Hence, $|U|$ is subharmonic  on $\Omega_U$ in a classical sense and the same argument also prove that $|U_j|$ is subharmonic  on $\Omega_{U_j}$ in a classical sense for any $j=1,\dots,n$.

Let $\varphi \in C^\infty_c(D)$, $\varphi \ge0$. Let us define
\begin{equation}
p_\e(x):=
\begin{cases}
0,  &\text{ if } x \in [0,\e/2],\\
\frac{1}{\e} (2x-\e),  &\text{ if } x \in [\e/2, \e],\\
1,  &\text{ if } x \in [\e, +\infty),\\
\end{cases}
\end{equation}
and $U_{\e,t}:= |U|-t p_\e(|U|)\varphi$  for any  $\e, t >0$. Then $\Omega_{U_{\e,t}} \subset \Omega_U $ and $U_{\e,t} \le |U|$. By subharmonicity of $|U|$ on $\Omega_U$,
\begin{equation}
\int_{D} |\nabla |U||^2 dx \le \int_{D} |\nabla U_{\e,t}|^2 dx.
\end{equation}
Furthermore,
\begin{equation}
|\nabla U_{\e,t}|^2= |\nabla |U||^2 -2t( |\nabla |U||^2 p_\e'(|U|) \varphi +p_\e(|U|)\nabla |U| \cdot \nabla \varphi) +o(t), \text{ as } t \to 0^+. 
\end{equation}
Hence,
\begin{equation}
\int_{D}p_\e(|U|)\nabla |U| \cdot \nabla \varphi \, dx \le -\int_{D} |\nabla |U||^2 p_\e'(|U|) \varphi dx \le 0,
\end{equation}
since $p_\e$ is increasing.
Passing to the limit as $\e  \to 0^+$, we conclude that $|U|$ is subharmonic in a distributional sense on $D$. The same argument also yields the distributional subharmonicity on $D$ of $U_j$ for any $j=1,\dots, n$.
\end{proof}

From the subharmonicity of $|U|$ for any $x \in D$, the maps
\begin{equation}
r \mapsto \fint_{B_r(x)} |U| \, dy \text{ and } r \mapsto \fint_{\partial B_r(x)} |U| \, d\mc{H}^{d-1} 
\end{equation}
are decreasing and so we may consider the representative of $|U|$ given by 
\begin{equation}
|U|(x):=\lim_{r \to 0^+}\fint_{B_r(x)} |U| \, dy=\lim_{r \to 0^+}\fint_{\partial B_r(x)} |U| \, d\mc{H}^{d-1}.
\end{equation}
In particular,  $U \in L^\infty_{loc}(D,\R^k)$.

\subsection{Lipschitzianity and non-degeneracy}\label{subsec_Lip}
We start with non-degeneracy which only needs testing the minimality of $U$ with comparators satisfying $\Omega_{W_j} \subset \Omega_{U_j}$ for any $j=1,\dots,n$.

\begin{proposition}{(Non-degeneracy)}\label{prop_non_dege}
There exists a dimensional constant $\kappa_d>0$ such that for any $j=1,\dots,n$,  $x \in \overline{\Omega}_{U_j}$ and any $r \in(0,d(x,\partial D))$,
\begin{equation}\label{ineq_non_dege}
\norm{|U_j|}_{L^\infty(B_r(x))} \ge \fint_{\partial B_r(x)} |U_j|\, d \mc{H}^{d-1}\ge \kappa_d \Lambda_j  r.
\end{equation}
Furthermore, 
\begin{equation}\label{ineq_nabla_U_j_trace}
\frac{1}{2}\int_{B_r(x)}|\nabla |U_j||^2 \,d x +\Lambda_j |\Omega_{U_j} \cap B_r(x)| \le C_d r^{-1}\norm{U_j}_{L^\infty(B_{2r}(x))}
\int_{\partial B_r} |U_j| \, d \mc{H}^{d-1},
\end{equation}
for some constant dimensional $C_d>0$, for any $x \in D$ and for any $r \in (0, d(x,\partial D)/2)$.
\end{proposition}
\begin{proof}
We note that \eqref{ineq_non_dege} can be proved as a simpler version of \cite[Lemma 2.6]{MTV_reg_spec}, which in turn is based on the analogous lemma from \cite{AC_one_phase}. Also  \eqref{ineq_nabla_U_j_trace} can be obtained following the proof of the same lemma, we repeat  some details for the sake of completeness.

Let us consider the radial function $\phi: B_2\setminus B_1$ solution to 
\begin{equation}
\begin{cases}
-\Delta \phi=0,    &\text{ in } B_2\setminus B_1,\\
\phi=0,  &\text{ on } \partial B_1,\\
\phi=1,  &\text{ on } \partial B_2.
\end{cases}
\end{equation}
Let $r \in (0, d(x,\partial D))/2$ and 
\begin{equation}
\phi_r:B_{2r}(x)\setminus B_r(x) \to \R, \quad \phi_r(x):=\norm{U_1}_{L^\infty(B_{2r}(x))}\phi(x/r). 
\end{equation}
Clearly on $\overline{B_{2r}(x)}\setminus B_r(x)$
\begin{equation}
|\nabla \phi_r(x)| \le C_d \norm{U_1}_{L^\infty(B_{2r}(x))} r^{-1},
\end{equation}
for some dimensional constant $C_d>0$.
Let $\tilde{U} \in H^1(D, \R^k)$ be for any $i \in A_1$
\begin{equation}
\tilde{u}_i=
\begin{cases}
\tilde{u}_i= \min\{u_i^+, \phi_r\}-\min\{u_i^-, \phi_r\}, &\text{ in } B_{2r}(x), \\
\tilde{u}_i=u_i, &\text{ in } D\setminus B_{2r}(x),
\end{cases}
\end{equation}
while $\tilde{u}_i=u_i$ for $ i \in A_j, j \neq 1$.
It is easy to see that  $\tilde{U} \in H^1(D, \Sigma)$. Hence, we may test the minimality of $U$ with $\tilde{U}$ thus obtaining
\begin{multline}
 \sum_{i  \in A_1} \int_{B_r(x)}|\nabla u_i|^2 \, dx + \sum_{j=1}^n \Lambda_j |\Omega_{U_j}\cap B_r(x)| 
\le \sum_{i  \in A_1} \int_{B_{2r}(x)\setminus B_r(x)}(|\nabla \tilde{u}_i|^2-|\nabla u_i|^2) \, dx\\
=\sum_{i  \in A_1} \int_{B_{2r}(x)\setminus B_r(x)}[-|\nabla \tilde{u}_i-\nabla u_i|^2+2\nabla \tilde{u}_i\cdot \nabla (\tilde{u}_i-u_i)] \, dx.
\end{multline}
Furthermore,
\begin{multline}
\sum_{i  \in A_1} \int_{B_{2r}(x)\setminus B_r(x)}\nabla \tilde{u}_i\cdot \nabla (\tilde{u}_i-u_i) \, dx\\
\sum_{i  \in A_1} \int_{B_{2r}(x)\setminus B_r(x)}\nabla \tilde{u}_i^+\cdot \nabla (\tilde{u}_i^+-u_i^+) \, dx
+\sum_{i  \in A_1} \int_{B_{2r}(x)\setminus B_r(x)}\nabla \tilde{u}_i^-\cdot \nabla (\tilde{u}_i^--u_i^-) \, dx\\
=-\sum_{i  \in A_1} \int_{B_{2r}(x)\setminus B_r(x)}\nabla\phi_r\cdot \nabla (u_i^+-\phi_r)^+ \, dx
-\sum_{i  \in A_1} \int_{B_{2r}(x)\setminus B_r(x)}\nabla\phi_r\cdot \nabla (u_i^--\phi_r)^- \, dx\\
\le C_d \norm{U_1}_{L^\infty(B_{2r}(x))} r^{-1} \sum_{i  \in A_1} \int_{\partial B_r(x)}  u_i \, dx.
\end{multline}
Since $|\nabla |U_1|| \le|\nabla U_1|$ in $D$, we have proved \eqref{ineq_nabla_U_j_trace} for $j=1$. We can argue in the same way to prove it for any other $j=2, \dots,n$.
\end{proof}

Now we turn to the Lipschitz continuity, which is far more delicate. Our argument is based on the multiphase Alt-Caffarelli-Friedmann monotonicity formula, see \cite{ACF_mono,CJK_mono,CTV_non_line,V_mono}, combined with an iteration scheme. 
\begin{theorem}\label{theo_multi_mono}
Let $m \ge 2$, $r_0>0$ and let $u_1,\dots,u_m \in H^1(B_{r_0}(x))$ be non-negative Sobolev functions such that in $B_{r_0}(x)$
\begin{equation}
-\Delta u_i \le 1  \text{ for any } i =1, \dots m, \quad \text{ and } \quad u_iu_j=0 \text{ for any } i,j =1, \dots m, i \neq j,
\end{equation}
where the inequality is meant in a distributional sense. Then there exists a  constant $C>0$, depending only  on $d$ and $r_0$,  and $\e\ge 0$ such that 
\begin{equation}\label{ineq_multi_mono}
\prod_{i=1}^m \frac{1}{r^{2+\e}}\int_{B_r(x)}\frac{|\nabla u_i|^2}{|x-y|^{d-2}}\, dy \le  C
\left(1+\sum_{i=1}^m \int_{B_{r_0}(x)}\frac{|\nabla u_i|^2}{|x-y|^{d-2}}\, dy \right)^m,
\end{equation}
with  $\e=0$ if $m=2$ while $\e>0$ if $m >2$.
\end{theorem}

Let us define for any $j=1,\dots, n$, any $x \in D$ and any $r \in (0,d(x,\partial D))$
\begin{equation}\label{def_Hjxr}
H_j(x,r):=\frac{1}{r}\fint_{\partial B_r(x)} |U_j|\, d \mc{H}^{d-1}.
\end{equation}
We are going to show that $H_j(x,r)$ is bounded in $r$ for any $x \in \partial \Omega_{U_j}$. We will make use of the following lemmas.

\begin{lemma}\label{lemma_sub}
Let $u \in H^1(B_{2r}(x))$ be non-negative and subharmonic. Then for any $y \in B_r(x)$ 
\begin{equation}\label{ineq_u_sub_means}
u(y) \le 2 \fint_{\partial B_{2r}(x)} u \, d \mc{H}^{d-1}.
\end{equation}
\end{lemma}

\begin{proof}
Let  $h$ in the unique solution to
\begin{equation}
\begin{cases}
-\Delta h=0, \text{ in } B_{2r}(x),\\
h=u, \text{ on } B_{2r}(x).
\end{cases}
\end{equation}
We can represent $h$  on $B_{2r}(x)$ as
\begin{equation}
h(y)=\frac{(2r)^2-|y|^2}{2\omega_d r} \int_{\partial B_{2r}(x)}\frac{u(\xi)}{|y-\xi|^d} \, d \mc{H}_\xi^{d-1}.
\end{equation}
For any $y \in B_r(x)$ and any $\xi \in \partial B_{2r}(x)$
\begin{equation}
\frac{(2r)^2-|y|^2}{\omega_d 2r|y-\xi|^d} \le \frac{2}{\omega_dr^{d-1}}.
\end{equation}
Hence, we conclude that for any $y \in B_r(x)$ 
\begin{equation}
h(y) \le 2 \fint_{\partial{B_{2r}(x)}}u \, d \mc{H}^{d-1}
\end{equation}
which yields \eqref{ineq_u_sub_means} by  subharmonicity.
\end{proof}

\begin{lemma}\label{lemma_sub_changing_points}
Let $u \in H^1(B_{R}(x))$ be non-negative and subharmonic. Let $y \in B_R(x)$ and let $r:=|x-y|$. Then 
\begin{equation}\label{ineq_sub_chang_point}
\fint_{\partial B_{\rho/4}(y)} u \, d \mc{H}^{d-1} \le C_d \fint_{\partial B_{\rho}(x)} u \, d \mc{H}^{d-1},
\end{equation}
for any $\rho \in [2r,R]$ and some dimensional constant $C_d>0$.
\end{lemma}

\begin{proof}
Since $B_{\rho-r}(y) \subset B_\rho(x)$ and $\rho-r \ge \rho/2$,  by subharmonicity
\begin{equation}
\int_{B_{\rho-r}(y)\setminus B_{\rho/4}(y)}u \, dz =d \omega_d\int_{\rho/4}^{\rho-r}t^{d-1}\fint_{\partial B_t(y)} u \, d \mc{H}^{d-1} 
\ge \omega_d[(\rho-r)^d-(\rho/4)^d]\fint_{\partial B_{\rho/4}(y)} u \, d \mc{H}^{d-1}
\end{equation}
while 
\begin{equation}
\int_{B_{\rho-r}(y)\setminus B_{\rho/4}(y)}u \, dz  \le \int_{B_{\rho}(x)}u \, dz \le \omega_d \rho^d\fint_{\partial B_\rho(x)} u \, d \mc{H}^{d-1}.
\end{equation}
Hence, we conclude that
\begin{equation}
\fint_{\partial B_{\rho/4}(y)} u \, d \mc{H}^{d-1} \le \frac{ \rho^d}{(\rho-r)^d-(\rho/4)^d}\fint_{\partial B_\rho(x)} u \, d \mc{H}^{d-1} \le 
 \frac{4^d}{2^d-1}\fint_{\partial B_\rho(x)} u \, d \mc{H}^{d-1},
\end{equation}
thus proving \eqref{ineq_sub_chang_point}.
\end{proof}

\begin{lemma}\label{lemma_iter}
Let $j=1,\dots, n$, $x \in \overline{\Omega}_{U_j}$ and $r  \in (0, d(x,\partial D)/4)$. Then
\begin{equation}\label{ineq_iter}
H_j(x,r/2)\le C_d\frac{|\Omega_{U_j} \cap B_r(x)|}{|B_r|}H_j(x,4r),
\end{equation}
for some dimensional constant $C_d>0$ and any $j=1, \dots, n$.
\end{lemma}
\begin{proof}
Let $x \in \overline{\Omega}_{U_j}$. 
In view of \eqref{ineq_non_dege} and  \eqref{ineq_nabla_U_j_trace}, we can follow the proof of \cite[Proposition 3.12]{BV_multi} to show that
for any $r \in  (0, d(x,\partial D)/4)$
\begin{equation}
\frac{\norm{U_j}_{L^\infty(B_{r/2}(x))}}{r/2}
\le C_d\frac{|\Omega_{U_j} \cap B_r(x)|}{|B_r|}\frac{\norm{U_j}_{L^\infty(B_{2r}(x))}}{2r}.
\end{equation}
By Proposition \ref{prop_existence_min}, $|U_j|$ is subharmonic in $D$ so by Lemma \ref{lemma_sub} 
\begin{multline}
\frac{\fint_{\partial B_{r/2}(x)} |U_j|\, d \mc{H}^{d-1} }{r/2}\le \frac{\norm{U_j}_{L^\infty(B_{r/2}(x))}}{r/2}\\
\le C_d\frac{|\Omega_{U_j} \cap B_r(x)|}{|B_r|}\frac{\norm{U_j}_{L^\infty(B_{2r}(x))}}{2r} 
\le  4C_d\frac{|\Omega_{U_j} \cap B_r(x)|}{|B_r|}\frac{\fint_{\partial B_{4r}(x)} |U_j|\, d \mc{H}^{d-1} }{4r}
\end{multline}
thus we have proved \eqref{ineq_iter}.
\end{proof}

\begin{proposition}\label{prop_multi_phases_means}
There exists $\e\ge0$ such that for any $\delta>0$ there exists a constant $C>0$, depending on $d$, $\norm{g}_{H^1(D,\R^k)}$ and $\delta$,  such that for any $B \subset \{1, \dots,n\}$ with $|B| \ge 2$ and  any $x \in D$ with $d(x,\partial D)>\delta$ there holds
\begin{equation}\label{ineq_multi_phases_means}
\prod_{j \in B}H_j(x,r) \le C r^{\e/2} \quad \text{ for any } r \in (0, \delta].
\end{equation}
Furthermore $\e=0$ if $|B|=2$ while $\e>0$ if $|B|>2$.
\end{proposition}

\begin{proof}
By Proposition \ref{prop_existence_min}, $|U_j|$ is subharmonic on $D$ for any $j=1, \dots n$ so that Theorem \ref{theo_multi_mono} holds for any $x \in D$ and 
any $r \in (0, d(x,\partial D))$ and any $B \subset \{1,\dots, n\}$ with at least two elements. Hence
\begin{multline}
\prod_{j \in B} \frac{1}{r^2}\int_{B_r(x)}\frac{|\nabla |U_j||^2}{|x-y|^{d-2}}\, dy \le  C
\left(1+\sum_{j \in B}\int_{B_{\delta}(x)}\frac{|\nabla |U_j||^2}{|x-y|^{d-2}}\, dy \right)^{|B|}r^\e \\
\le C\left(1+\sum_{j \in B}\int_{D}\frac{|\nabla |U_j||^2}{|x-y|^{d-2}}\, dy \right)^{|B|} r^\e\le Cr^\e,
\end{multline}
for some constant $C>0$ depending on $\delta, d$ and $\norm{g}_{H^1(D,\R^k)}$.
We recall  the following  trace inequality for any non-negative $u \in H^1(B_r(x))$ and any $r>0$,
\begin{equation}
\frac{1}{r^2}|\{u=0\}|\left(\fint_{\partial B_r(x)} u \, d \mc{H}^{d-1}\right)^2 \le C_d\int_{B_r(x)}|\nabla(u-h)|^2 dx,
\end{equation}
where $C_d$ is a positive dimensional constant and $h$ is the unique solution of 
\begin{equation}
\begin{cases}
-\Delta h=0, \text{ in } B_r(x),\\
h=u, \text{ on } B_r(x),
\end{cases}
\end{equation}
see for example in  \cite[Lemma 3.7]{V_book}. Since
\begin{equation}
\int_{B_r(x)}|\nabla(u-h)|^2 dx= \int_{B_r(x)}(|\nabla u|^2-|\nabla h|^2) dx \le \int_{B_r(x)}|\nabla u|^2 \, dx,
\end{equation}
we obtain, by \eqref{def_Hjxr},
\begin{multline}
\prod_{j \in B}\frac{1}{r^d}|\{|U_j|=0 \cap B_r(x)\}||H_j(x,r)|^2 \\
\le C_d\prod_{j \in B} \frac{1}{r^d}\int_{B_r(x)}|\nabla |U_j||^2\, dy 
\le  CC_d\prod_{j \in B} \frac{1}{r^2}\int_{B_r(x)}\frac{|\nabla |U_j||^2}{|x-y|^{d-2}}\, dy  \le  CC_d r^{\e}.
\end{multline}
Furthermore,
\begin{multline}
|\{U_j=0\} \cap B_r(x)|=|B_r|-|\Omega_{U_j}\cap B_r(x)|=|\{U=0\}\cap B_r(x)|+\sum_{j=1, j' \neq j}^m|\Omega_{U_{j'}} \cap B_r(x)|\\
 \ge \sum_{j=1, j' \neq j}^m|\Omega_{U_{j'}} \cap B_r(x)|.
\end{multline}
Let now consider only subsets $B\subset\{1,\dots,n\}$  such that $H_j(x,4r)>0$ for any $j \in B$ and $|B| \ge 2$. By Lemma \ref{lemma_iter}, for any such $B$ it follows that 
\begin{equation}
\prod_{j \in B}|H_j(x,r)|^2\sum_{j'\in B,j'\ne j} \frac{H_{j'}(x,r/2)}{H_{j'}(x,4r)} \le  Cr^\e,
\end{equation}
for some constant $C>0$ depending on $\delta, d$ and  $\norm{g}_{H^1(D,\R^k)}$. By subharmonicity of $|U_j|$, 
\begin{equation}
H_j(x,r) =\frac{1}{r}\fint_{\partial B_r(x)} |U_j| \, d \mc{H}^{d-1} \ge \frac{1}{2}\frac{1}{r/2}\fint_{\partial B_{r/2}(x)} |U_j| \, d \mc{H}^{d-1}
=\frac{1}{2}H_j(x,r/2),
\end{equation}
thus 
\begin{equation}
\prod_{j \in B}|H_j(x,r/2)|^2\sum_{j'\in B,j'\ne j} \frac{H_{j'}(x,r/2)}{H_{j'}(x,4r)} \le 4 C r^\e.
\end{equation}
Furthermore, by the arithmetic mean versus geometric mean inequality,
\begin{equation}
\sum_{j'\in B,j'\ne j} \frac{H_{j'}(x,r/2)}{H_{j'}(x,4r)} \ge (|B|-1)\left(\prod_{j'\in B,j'\ne j} \frac{H_{j'}(x,r/2)}{H_{j'}(x,4r)}\right)^{\frac{1}{|B|-1}}.
\end{equation}
Hence, for any $B \subset \{1,\dots,n\}$ with  $|B| \ge 2$
\begin{equation}
\prod_{j \in B}|H_j(x,r/2)|^3 \le 4 C r^\e \prod_{j \in B} H_{j}(x,4r).
\end{equation}
Indeed, we have just proved it if $H_{j}(x,4r)>0$ for any $j \in B$ while if $H_{j}(x,4r)=0$ for some $j \in B$ then also  $H_j(x,r/2)=0$ by \eqref{def_Hjxr} and subharmonicity which makes the above inequality trivial. Letting
\begin{equation}
A_B(x,r):=\frac{1}{2\sqrt{Cr^\e}}\prod_{j \in B}H_j(x,r),
\end{equation}
we have shown that for  any $x \in D$,  $r \in (0,d(x,\partial D)/8)$, and  $B \subset \{1,\dots,n\}$ with  $|B| \ge 2$
\begin{equation}\label{proof_prop_multi_phases_means_1}
A_B(x,r) \le |A_B(x,8r)|^{1/3}.
\end{equation}
Let $x \in D$ with $d(x,\partial D)> \delta$. For any $r \in  [\delta/8,\delta]$, by the trace inequality on $B_r(x)$
\begin{equation}
A_B(x,r) =\frac{1}{2\sqrt{C r^\e}}\prod_{j \in B} \frac{1}{d \omega_d r^d}\int_{\partial B_r(x)} |U_j| \, d \mc{H}^{d-1} \le C
\end{equation}
for some constant $C>0$ depending on $\delta, d$ and $\norm{g}_{H^1(D,\R^k)}$.

Let now $r \in (0, \delta/8)$. There exists $k \in \mathbb{N}$ such that $\delta/8 \le8^kr\le \delta$. Iterating \eqref{proof_prop_multi_phases_means_1} $k$ times we obtain
\begin{equation}
A_B(x,r) \le |A_B(x,8^k r)|^{3^{-k}} \le C^{3^{-k}} \le \max\{1,C\}^{3^{-k}}.
\end{equation}
More precisely, 
\begin{equation}
k=\lfloor \log_8(\delta/r) \rfloor \ge  \log_8(\delta/r)-1,
\end{equation}
where $=\lfloor \cdot \rfloor $ denote the integer part of a real number. Hence, since $r \in (0,\delta/8)$,
\begin{equation}
3^{-k} \le 3^{-\log_8(\delta/r)+1}=3 (r/\delta)^{\log_83} \le 3 (1/8)^{\log_83}.
\end{equation}
We conclude that 
\begin{equation}
A_B(x,r)  \le \max\{1,C\}^{3 (1/8)^{\log_83}},
\end{equation}
thus proving \eqref{ineq_multi_phases_means}.
\end{proof}

For any $\delta >0$ let 
\begin{equation}\label{def_Ddelta}
D_\delta:=\{x \in D: d(x,\partial D)> \delta\}.
\end{equation}

\begin{proposition}[Lipschitzianity]\label{prop_lip}
For any $i=1,\dots,k$, $u_i \in C^{0,1}_{loc}(D)$. Furthermore,
\begin{equation}
\norm{\nabla  u_i}_{L^\infty(D_\delta)} \le C,
\end{equation}
where $C>0$ depending on $\delta, d$, $\Lambda_j$ for $j=1, \dots,n$ and $\norm{g}_{H^1(D,\R^k)}$.
\end{proposition}

\begin{proof}
Since $u_i$ is harmonic on $\Omega_{U_{j_i}}$ by Proposition \ref{prop_existence_min}, 
it is enough to show that for any $x \in \partial \Omega_{U_{j_i}} \cap D_\delta$ and any $i=1, \dots. k$
\begin{equation}
\fint_{\partial B_r(x)} |u_i| \, d \mc{H}^{d-1} \le  C \quad \text{ for any } r \in (0, \delta),
\end{equation}
see for example \cite[Lemma 3.5]{V_book}. 
Equivalently, we need to show that for any $j=1, \dots, n$ and $x \in \partial \Omega_{U_j} \cap D_\delta$
\begin{equation}\label{proof_prop_lip_1}
 H_j(x,r)  \le  C \quad \text{ for any } r \in (0, \delta).
\end{equation}
Let $j \in \{1, \dots, n\}$ and $x \in \partial \Omega_{U_j}$. We distinguish two cases.

\textbf{Case 1.} There exists $j' \neq j$ such that $x \in  \Omega_{U_{j'}}$. 

By Propositions \ref{prop_non_dege}-\ref{prop_multi_phases_means} with $B=\{j,j'\}$
\begin{equation}
C \ge  H_j(x,r)  H_{j'}(x,r) \ge  H_j(x,r)  \kappa_d \Lambda_{j'} 
\end{equation}
thus \eqref{proof_prop_lip_1} holds in this case.

\textbf{Case 2.} There exist $r>0$ such that  $B_r(x) \cap \Omega_{U_{j'}}=\emptyset$ for any $j '\neq j$.

Let $r(x)$ be the biggest $r>0$ such that $B_{r(x) }(x)\cap \Omega_{U_{j'}}=\emptyset$ for any $j '\neq j$. Then $U_j$ is a  minimizer the  vectorial Bernoulli free boundary functional in  $B_{r(x)}(x)$, that is, the functional
\begin{equation}
J_{\Lambda_j}(W,B_{r(x)}(x)):=\int_{B_{r(x)}(x)} |\nabla W|^2\, dx +  \Lambda_j\left|\Omega_{W}\right|,
\end{equation}
with $W \in H^1(B_{r(x)}(x),\R^{k_j})$, $W-U_j \in H_0^1(B_{r(x)}(x),\R^{k_j})$.  Then, for example by \cite{BMPV_Lip},
\begin{equation}
 H_j(x,r)  \le  C \quad \text{ for any } r \in (0, r(x)).
\end{equation}
By definition of $r(x)$ there exist $j' \in\{1,\dots, n\}$  and $y \in  \Omega_{U_{j'}}$ such that $|x-y|=r(x)$.
Furthermore, for any $r \in (2r(x),\delta)$ by Proposition \ref{prop_multi_phases_means}  with $B=\{j,j'\}$, Lemma \ref{lemma_sub_changing_points} and Proposition \ref{prop_non_dege}
\begin{equation}
C \ge  H_j(x,r)  H_{j'}(x,r) \ge C_d H_j(x,r)  H_{j'}(y,r/4) \ge  H_j(x,r)   C_d \kappa_d \Lambda_{j'} 
\end{equation}
so that 
\begin{equation}
 H_j(x,r)  \le  C \quad \text{ for any } r \in (2r(x),\delta)
\end{equation}
as well.
By subharmonicity for any $r \in [r(x), 2r(x)]$
\begin{equation}
 H_j(x,r) \le  2 H_j(x,2r(x))  \le 2C
\end{equation}
thus proving \eqref{proof_prop_lip_1}.
\end{proof}

A  simple and interesting consequence of Proposition \ref{prop_multi_phases_means} is the absence of triple points.
\begin{corollary}[Absence of triple points]\label{corol_no_triple}
For any distinct $j,j'j'' \in\{1, \dots, n\}$
\begin{equation}
\partial \Omega_{U_j} \cap \partial \Omega_{U_{j'}} \cap \partial \Omega_{U_{j''}}\cap D= \emptyset.
\end{equation}
\end{corollary}

\begin{proof}
Let $x \in \partial \Omega_{U_j} \cap \partial \Omega_{U_{j'}} \cap \partial \Omega_{U_{j''}}\cap D$. 
By Propositions \ref{prop_non_dege}-\ref{prop_multi_phases_means} with $B:=\{j,j'j''\}$,
\begin{equation}
\kappa_d^3 \Lambda_j  \Lambda_{j'} \Lambda_{j''}\le H_j(x,r) H_{j'}(x,r)H_{j''}(x,r) \le C r^{\e/2}
\end{equation}
for any $r \in (0,d(x,\partial D))$, a contradiction since $\e>0$.
\end{proof}

By non-degeneracy and Lipschitzianity we can obtain a lower density estimate in standard way, see for example \cite[Lemma 5.1]{V_book}.
\begin{corollary}[Lower density estimates]\label{corol_lower_dens}
Let $j \in\{1, \dots, n\}$ and let $x \in  \partial \Omega_{U_j}\cap D$. Then, there exists a constant $\kappa>0$, depending only on $d$, $\Lambda_{j'}$ for $j'=1, \dots,n$ and $\norm{g}_{H^1(D,\R^k)}$ such that 
\begin{equation}\label{ineq_lower_dens}
|B_r(x) \cap \partial \Omega_{U_j}| \ge \kappa |B_r|  \quad \text{ for any } r \in (0,d(x,\partial D)).
\end{equation}
\end{corollary}
At two-phase or branching points we can turn \eqref{ineq_lower_dens} into an upper density estimate for each of the phases $\Omega_{U_j}$.
\begin{corollary}[Upper density estimates at two-phase and branching points]\label{corol_upper_esti}
Let $j,j' \in\{1, \dots, n\}$ with $j \neq j'$ and let $x \in  \partial \Omega_{U_j}\cap \partial \Omega_{U_{j'}}\cap D$. Then, there exists a constant $\kappa>0$, depending only on $d$, $\Lambda_{j''}$ for $j''=1, \dots,n$ and $\norm{g}_{H^1(D,\R^k)}$ such that 
\begin{equation}\label{ineq_upper_dens}
|B_r(x) \cap \partial \Omega_{U_j}| \le (1-\kappa) |B_r|  \quad \text{ and  } \quad |B_r(x) \cap \partial \Omega_{U_{j'}}| \le (1-\kappa) |B_r|
\end{equation}
for any $r \in (0,d(x,\partial D))$.
\end{corollary}

\begin{proof}
For any for any $r \in (0,d(x,\partial D))$, by \eqref{ineq_lower_dens},
\begin{multline}
|\Omega_{U_j}\cap B_r(x)|=|B_r|-|\{U=0\}\cap B_r(x)|-\sum_{j=1, j'' \neq j}^m|\Omega_{U_{j''}} \cap B_r(x)|\\
 \le |B_r| - | \Omega_{U_{j'}} \cap B_r(x)| \le (1-\kappa)|B_r|,
\end{multline}
thus proving \eqref{ineq_upper_dens}.
\end{proof}

\section{Weiss Monotonicity Formula and Blow-ups}\label{sec_blowup_Weiss}
After establishing basic properties as existence,  non-degeneracy and Lipschitzianity in the previous section, now we turn our attention to the  regularity of the free boundary $\partial \Omega_U$ in $D$. 
To this end, given $x_0 \in \partial \Omega_U\cap D$, we may localize the problem in a neighbourhood $B_r(x_0)$ for some $r>0$. 
Furthermore, there exists $r_0>0$ such that $B_{r_0}(x_0) \cap  \Omega_{U_{j''}}=\emptyset$ for any $j'' \neq j,j'$ in view of Corollary \ref{corol_no_triple}.
Hence, we may suppose that $j=1$ and $j'=2$ for the sake of simplicity.  We also note that $(U_1,U_2)$ minimizes \eqref{def_J} in the domain $B_{r_0}(x_0)$ with boundary datum $(U_1,U_2)$. Therefore,  we can make use all the notations of Section \ref{sec_intro} for this particular case, that is, $n=2$ and $D=B_{r_0}(x_0)$. For example $U=(U_1,U_2)=(u_1,\dots,u_k)$ denotes a minimizer of \eqref{def_J} in $B_{r_0}(x_0)$ with $n=2$ for the rest of the paper.

We also denote $B_r(0)$ simply with $B_r$ for any $r>0$.

\subsection{Blow-up limits}\label{sub_blow_ups}
For any $x,y \in \R^d$  and $r>0$ let us define 
\begin{equation}
V_{r,x}(y):=\frac{1}{r}V(x+ry).
\end{equation}
Let
\begin{equation}
r_h \to 0^+, \quad   \{x_h\} \subset \partial \Omega_{U_1} \cup \partial \Omega_{U_2} \text{ with } x_h \to x \text{ and }x \in \partial \Omega_{U_1} \cup \partial \Omega_{U_2}
\end{equation}
as $h \to \infty$. 
 Since $U_j$ is Lipschitz there exists
\begin{equation}\label{eq_blow_up}
U_{j,0} \in C_{loc}^{0,1}(\R^d,\R^{k_j}) 
\quad \text{ such that }\quad(U_j)_{r_h,x_h} \to U_{j,0}
\end{equation}
uniformly in $B_R$ for any $R>0$ as $h \to \infty$ and up to subsequences for $j=1,2$. 

Note that, since we are considering a sequence of points $\{x_h\}$, it may well happen that $U_{j,0}\equiv0$ for some $j=1,2$ which is instead prevented by the non-degeneracy of $U_j$ if for example $x_h=x$ for any $h \in \mb{N}$ and $x \in \partial \Omega_{U_1} \cap \partial \Omega_{U_2}$. Furthermore, $U_{j,0}$ might not be unique and may depend on the sequence $x_h$ and $r_h$.

\begin{definition}\label{def_blow_ups}
A sequence $(U_j)_{r_h,x_h}$ is called \textit{blow-up sequence} and any limit $U_{j,0}$ of some blow-up sequence is called \textit{blow-up limit}. 
Furthermore, we let $\mc{BU}_{U_j}(x)$ be the set of all the blow up limits of $U_j$ at $x$, that is, limits  obtained as in \eqref{eq_blow_up} for some  sequence $x_h \to x$.
\end{definition}
We are considering blow up sequences with variable centers $\{x_h\}$ in $\partial \Omega_{U_1} \cup \partial \Omega_{U_2}$ instead of simply a fixed point $x$. This is crucial for our approach to the regularity of $\partial \Omega_{U_j}$ which is based on proving flatness in the sense of Reifenberg, see in particular the proof of Proposition \ref{prop_Reif_flat}.

We start by proving strong convergence in $H^1_{loc}(\R^d,\R^{k_j})$.
\begin{proposition}\label{prop_blow_up_limit_H1}
For $j=1,2$,  any blow-up sequence and any $R>0$
\begin{equation}\label{eq_strong_limit_H1}
{(U_j)}_{r_h,x_h}  \to U_{j,0} \quad \text{ strongly in } H^1(B_R,\R^{k_j}) \quad \text{ as } h \to \infty.
\end{equation}
Furthermore, $u_{i,0}$ is harmonic in $\Omega_{U_{j_i,0}}$ for any $i=1,\dots, k$.
\end{proposition}

\begin{proof}
By Proposition \ref{prop_existence_min}, $u_i$ is harmonic over $\Omega_{U_j}$. Hence 
also $(u_i)_{r_h,x_h}$ is harmonic over $\Omega_{(U_j)_{r_h,x_h}}$. If $y \in \Omega_{U_{j,0}}$, then, since $u_{i,0}$ is continuous, there exists $r>0$ such that $u_{i,0}>0$ over $\overline{B}_r(y)$. For $h$ big enough, also  $(u_i)_{r_h,x_h}>0$ over $\overline{B}_r(y)$. Otherwise there would be a sequence of points $z_h \in \overline{B}_r(y)$ converging, up to subsequences, to some $z \in \overline{B}_r(y)$ such that $(u_i)_{r_h,x_h}(z_h)=0$ while 
$u_{i,0}(z)>0$, a contradiction with the uniform convergence of $(u_i)_{r_h,x_h}$ to $u_{i,0}$. In particular, $u_{i,0}$ is harmonic over 
$\Omega_{{U_{j,0}}}$.

We  now follows the strategy of \cite[Step 5 proof Theorem 3.1]{RTT_eigen}, see also \cite[Lemma 5.2]{MST_spectral_partition} and  \cite[Lemma 7.4]{MTV_spect_2}.
We already know that ${(U_j)}_{r_h,x_h} \rightharpoonup {U_{j,0}}$ weakly in $H^1(B_R,\R^{k_j})$ as $h \to \infty$  for any $R>0$.
Furthermore,  by \cite[Proposition 2.3]{BMPV_Lip}, for any $R>0$, $h$ large enough and $\varphi \in H^1(B_R)$
\begin{equation}\label{proof_prop_blow_up_1}
-\int_{B_R}\nabla (u_i)_{r_h,x_h} \cdot \nabla \varphi \, dy= \int_{B_R} \varphi \, d\mu_h \quad \text{ and } -\int_{B_R}\nabla u_{i,0} \cdot \nabla \varphi \, dy= \int_{B_R} \varphi \, d\mu 
\end{equation}
for some positive Radon measures $\mu_h, \mu$ concentrated on $\partial \Omega_{{(U_j)}_{r_h,x_h}} \cap B_R$ and $\partial \Omega_{U_{j,0}}\cap B_R$ respectively.
Since $(u_i)_{r_h,x_h} $ are equi-Lipschitz,   taking $\varphi \ge 0$ with $\varphi \equiv 1$ on $B_R$
\begin{equation}
\mu_h(B_R) \le -\int_{B_{2R}}\nabla (u_i)_{r_h,x_h} \cdot \nabla \varphi \, dy \le C,
\end{equation}
for some positive constant that does not depend on $h$. Hence,  again with $\varphi \ge 0$ and  $\varphi \equiv 1$ on $B_R$,
\begin{multline}
\int_{B_R}|\nabla ((u_i)_{r_h,x_h}-u_{i,0})|^2 \, dy \le \int_{B_{2R}}|\nabla ((u_i)_{r_h,x_h}-u_{i,0})|^2 \varphi \, dy\\
=-\int_{B_{2R}}\nabla[(u_i)_{r_h,x_h}-u_{i,0}] \cdot \nabla \varphi[(u_i)_{r_h,x_h}-u_{i,0}] \, dy\\
-\int_{B_{2R}}[(u_i)_{r_h,x_h}-u_{i,0}]  \varphi  \, d\mu_h + \int_{B_{2R}}[(u_i)_{r_h,x_h}-u_{i,0}]  \varphi  \, d\mu,
\end{multline}
by \eqref{proof_prop_blow_up_1} tested with $[(u_i)_{r_h,x_h}-u_{i,0}] \varphi$. Thanks to the uniform convergence of $(u_i)_{r_h,x_h}$ to $u_{i,0}$, we can pass to the limit as $h \to \infty$ thus proving \eqref{eq_strong_limit_H1}.
\end{proof}

In order to show the convergence of the not-zero  sets of a blow up sequence to the not-zero  set of the corresponding blow up limit in an $L^1$ sense we follow the approach of \cite[Lemma 7.4]{MTV_spect_2}. In order to adapt it to our setting,  we need a preliminary lemma.
\begin{lemma}\label{lemma_few_density_1}
Let  $U_{j,0} \in \mc{BU}_{U_j}(x)$ for $j=1,2$. Then $\mc{L}^d$-a.e. $y \in  \partial \Omega_{U_{j,0}}$ is a not a point of density $1$ for $\Omega_{U_{j,0}}$, that is,
\begin{equation}
\left|\left\{y \in \partial \Omega_{U_{j,0}}:\lim_{\rho\to0^+} \frac{|\Omega_{U_{j,0}} \cap B_\rho(y)|}{|B_\rho|}=1\right\}\right|=0.
\end{equation}
\end{lemma}

\begin{proof}
Let $y \in  \partial \Omega_{U_{1,0}}$ and suppose that the density of $\Omega_{U_{1,0}}$ at $y$ is $1$. 
We claim that there exists a radius $r(y)$ and  $h_0 \in \mb{N}$ such that 
\begin{equation}
B_{r(y)}(y) \cap \Omega_{(U_2)_{r_h,x_h}}= \emptyset
\end{equation}
for any $h \ge h_0$. We argue by contradiction supposing that there exists a sequence $y_n \to y$ such that $(U_2)_{r_h,x_h}(y_n)\neq0$. For any small $\rho>0$,
by Proposition \ref{prop_non_dege}, $\norm{(U_2)_{r_h,x_h}}_{L^\infty(B_\rho(y_n))} \ge \eta \rho$ for some $\eta>0$ that does not depend on $h$.
Passing to the limit as $n \to \infty$ and $h \to \infty$ we obtain $\norm{U_{2,0}}_{L^\infty(B_{2\rho}(y))} \ge \eta \rho$. Let $\tilde{y}$ be such that
$|U_{2,0}|(\tilde{y})=\norm{U_{2,0}}_{L^\infty(B_{2\rho}(y))}$.
Hence, from the Lipschitzianity of $|U_{2,0}|$,  we obtain
\begin{equation}
|U_{2,0}|(z) \ge  L |z-\tilde{y}|-|U_{2,0}|(\tilde{y}) \ge \eta \rho - L|z-\tilde{y}|
\end{equation}
so that $|U_{2,0}| \neq 0$ in $B_{\eta\rho/L}(\tilde{y})\subset B_{3 \rho}(y)$. Since $U_{1,0}U_{2,0}\equiv0$, then we have reached a contradiction with the fact that $y$ has density $1$ for  $\Omega_{U_{2,0}}$.

In $B_{r(y)}(y)$, we then have that $(U_1)_{r_h,x_h}$  is a  minimizer of the  vectorial Bernoulli free boundary functional for any $h\ge h_0$. Hence, so is  $U_{1,0}$, see \cite{MTV_reg_vect}. Let us define the open set
\begin{multline}
A:=\{z \in \partial \Omega_{U_{1,0}}:U_{1,0}  \text{ minimizes the  vectorial Bernoulli free boundary }\\
\text{ functional in $B_r(z)$ for some } r>0  \}.
\end{multline}
Let $Q$ be a countable dense subset of $A$. Then, 
\begin{equation}
A=\bigcup_{q \in Q}B_{r(q)}(q), \quad \text{ where } r(q):=\sup\{r>0: B_{2r(q)}(q)\subset A\}. 
\end{equation}
Hence, 
\begin{equation}
\left|\left\{y \in \partial \Omega_{U_{1,0}}:\lim_{\rho\to0^+} \frac{|\Omega_{U_{1,0}} \cap B_\rho(y)|}{|B_\rho|}=1\right\}\right| 
\le \sum_{q \in Q}|\partial \Omega_{U_{1,0}} \cap B_{r(q)} |=0,
\end{equation}
which conclude the proof for $j=1$. The case $j=2$ is of course identical.
\end{proof}

\begin{proposition}\label{prop_blow_up_limit_notzero_set_L1}
Let $\{x_h\} \subset \partial \Omega_{U_1} \cap \partial \Omega_{U_2}$. For $j=1,2$ and any $R>0$
\begin{equation}\label{eq_blow_up_limit_notzero_set_L1}
\chi_{\Omega_{{(U_j)}_{r_h,x_h}}}  \to \chi_{\Omega_{U_{j,0}}} \quad \text{ strongly in } L^1(B_R) \quad \text{ as } h \to \infty.
\end{equation}
In particular, 
\begin{align}
 &\overline{\Omega}_{{(U_j)}_{r_h,x_h}} \to \overline\Omega_{U_{j,0}}, \quad \text{ and } \quad \R^d \setminus\Omega_{{(U_j)}_{r_h,x_h}} \to  \R^d \setminus\Omega_{U_{j,0}}, \label{eq_hausdorf_full_sets}\\
 &\partial \Omega_{{(U_j)}_{r_h,x_h}} \to \partial \Omega_{U_{j,0}},   \label{eq_hausdorf_boundary}
\end{align}
locally in Hausdorff sense as $h \to \infty$.
\end{proposition}

\begin{proof}
By the dominated convergence theorem, it is enough to prove pointwise convergence for a.e. $y \in \R^d$.
Furthermore, a.e. $y \in \R^d$ has ether density $0$ or $1$ for  $\Omega_{U_{j,0}}$.

If $y$ has density $1$, in view of Lemma \ref{lemma_few_density_1}, we may suppose that for  $y \in \Omega_{U_{j,0}}$. For $h$ large enough, the uniform convergence of ${(U_j)}_{r_h,x_h}$ to $U_{j,0}$ implies that also ${(U_j)}_{r_h,x_h}(y) \neq 0$ so that 
$\chi_{\Omega_{U_{j,0}}}(y)= \lim_{h \to \infty}\chi_{\Omega_{{(U_j)}_{r_h,x_h}}}(y)$.

If instead $y$ has density $0$, the continuity of $U_{j,0}$ yields $U_{j,0}(y)=0$. Suppose by contradiction that ${(U_j)}_{r_h,x_h}(y) \neq 0$ for any $h$.
Then, by Proposition \ref{prop_non_dege},  $\norm{(U_j)_{r_h,x_h}}_{L^\infty(B_\rho(y))} \ge \eta \rho$ for some $\eta>0$ that does not depend on $h$ and any small $\rho>0$. It follows that also $\norm{U_{j,0}}_{L^\infty(B_\rho(y))} \ge \eta \rho$ by uniform  convergence of $(U_j)_{r_h,x_h}$ to $U_{j,0}$. Hence, since $U_{j,0}$ is Lipschitz, it is standard  (see for example the proof of \cite[Lemma 5.1]{V_book})  to deduce that the density of  $y$ in $\Omega_{U_{j,0}}$ is strictly positive, a contradiction. 
It follows that ${(U_j)}_{r_h,x_h}(y)=0$ for $h$ large enough  which yields \eqref{eq_blow_up_limit_notzero_set_L1}.

It is a classical fact that \eqref{eq_blow_up_limit_notzero_set_L1} together with uniform density estimates, that is \eqref{ineq_lower_dens} and \eqref{ineq_upper_dens}, yield \eqref{eq_hausdorf_full_sets}. Furthermore, \eqref{eq_hausdorf_boundary} is a consequence of  \eqref{eq_hausdorf_full_sets}.
\end{proof}

For any any $R>0$ and any $V\in H^1(B_R,\Sigma)$ we define  the shape variation  $\delta \mc{F}_{\Lambda}(V)$ as
\begin{equation}\label{def_deltaF0}
\delta\mc{F}_\Lambda(V) [\xi]:= \sum_{i=1}^k \int_{B_R} [-2\nabla  v_i\cdot D\xi   \nabla  v_i+|\nabla  v_i|^2 \dive(\xi)] \, dx
+\Lambda_1\int_{\Omega_{V_1}}\dive(\xi) \, dx +\Lambda_2\int_{\Omega_{V_2}}\dive(\xi) \, dx
\end{equation}
for any $\xi \in C^1_c(B_R,\R^d)$.

\begin{lemma}\label{lemma_shape_der_blow_up}
Any minimizer $U$ of \eqref{prob_min} is stationary with respect to shape derivatives, that is, 
\begin{equation}\label{eq_shape_der}
\delta\mc{F}_\Lambda(U) [\xi]=0 \quad \text{ for any } \xi \in C^1_c(B_{r_0},\R^d).
\end{equation}
Furthermore, let  $U_{j,0} \in \mc{BU}_{U_j}(x)$ for $j=1,2$ and some sequence  $\{x_h\} \subset \partial \Omega_{U_1} \cap \partial \Omega_{U_2}$. Then also
\begin{equation}\label{eq_shape_der_blow_up}
\delta\mc{F}_\Lambda(U_{1,0},U_{2,0}) [\xi]=0 \quad \text{ for any } \xi \in C^1_c(B_R,\R^d) \text{ and } R>0.
\end{equation}
\end{lemma}
\begin{proof}
For any $\xi \in C_c^\infty( B_{r_0}, \R^d)$ and  any $t>0$  small enough  $\Psi_t(y):= y +t \xi(y)$  is a diffeomorphism. 
If we let  $U_{j,t}:=U_j \circ \Psi_t^{-1}$ for $j=1,2$, then $U_t:=(U_{1,t},U_{2,t})$ is a competitor for $U$ since $U_{1,t}U_{2,t}=0$ in $B_{r_0}$.
The minimality of $U$ yields \eqref{eq_shape_der}, see for example \cite[Lemma 9.5]{V_book}. 

With a change of variable for any $R>0$ and $h$ large enough, \eqref{eq_shape_der_blow_up} holds for $(U_j)_{r_h,x_h}$. In view of Propositions \ref{prop_blow_up_limit_H1}-\ref{prop_blow_up_limit_notzero_set_L1}, we may prove \eqref{eq_shape_der_blow_up} passing to the limit as $h \to \infty$.
\end{proof}

\subsection{Weiss monotonicity formula}\label{sub_Weiss}
Let for any $x \in \R^d$ and $r>0$
\begin{equation}\label{def_Weiss}
W_\Lambda(r,x,V):=\frac{1}{r^d} \left(\int_{B_r(x)} |\nabla V|^2 \, dy +\Lambda|\Omega_V\cap B_r(x)|\right)-\frac{1}{r^{d+1}}\int_{\partial B_r(x)} |V|^2 \, d \mc {H}^{d-1}
\end{equation}
for any $V \in H^1(B_r(x), \R^m)$ for some $m \in \mathbb{N}\setminus\{0\}$.
It is a classical computation to see that 
\begin{equation}
\pd{W_\Lambda}{r}(r,x,V)= \frac{d}{r}\left(W_\Lambda(1,0,(Z_V)_{r,x})-W_\Lambda(1,0,V_{r,x})\right)
+\frac{1}{r}\int_{\partial B_1} |y \cdot \nabla V_{r,x} -V_{r,x}|^2 \, d\mc{H}^{d-1},
\end{equation}
where $(Z_V)_{r,x}$ is the one homogeneous replacement of $V$ on $B_1$, that is, 
\begin{equation}
(Z_V)_{r,x}:B_1 \to \R^{m}, \quad (Z_V)_{r,x}:=|x| V(x/|x|),
\end{equation}
we refer to \cite[Chapter 9]{V_book}.

\begin{proposition}\label{prop_mono_Weiss}
For any $x \in \partial\Omega_{U_1} \cap \partial\Omega_{U_2}$ and any $r \in (0,r_0-|x-x_0|)$ the function
\begin{equation}
\Phi_{\Lambda}(r,x,U_1,U_2):=W_{\Lambda_1}(U_1,r,x)+W_{\Lambda_2}(U_2,r,x)
\end{equation}
is monotone in $r$ and continuous in $x$. Furthermore, for a.e. $r \in (0,r_0-|x-x_0|)$
\begin{equation}\label{eq_equi_part_energy}
\pd{\Phi}{r}(r,x,U_1,U_2)= \frac{2}{r}\sum_{j=1}^2\int_{\partial B_1} |y \cdot \nabla (U_j)_{r,x} -(U_j)_{r,x}|^2 \, d\mc{H}^{d-1}(y)
\end{equation}
and finally
\begin{equation}\label{eq_Weiss_density}
\lim_{r \to 0^+}\Phi_{\Lambda}(r,x,U_1,U_2)=\frac{\omega_d}{2}(\Lambda_1 + \Lambda_2).
\end{equation}
\end{proposition}

\begin{proof}
In view of \eqref{eq_shape_der}, we can easily prove  \eqref{eq_equi_part_energy}  taking into account \cite[Lemma 9.8]{V_book}. The continuity with respect to $x$ is also simple. Furthermore, the monotonicity of $\Phi_{\Lambda}$ in $r$ implies that any $U_{j,0} \in \mc{BU}_{U_j}(x)$ limit of a blow-up sequence of the form ${(U_j)}_{r_h,x}$ is one-homogeneous. 
Indeed, for any $\rho>0$
\begin{equation}
\Phi_{\Lambda}(\rho,0,U_{1,0},U_{2,0})=\lim_{h \to \infty}\Phi_{\Lambda}(\rho,0,(U_1)_{r_h,x},(U_2)_{r_h,x})
=\lim_{h \to \infty}\Phi_{\Lambda}(\rho r_h,x,U_1,U_2).
\end{equation}
Hence, $\Phi_{\Lambda}(\rho,0,U_{1,0},U_{2,0})$ is constant in $\rho$. By \eqref{eq_shape_der_blow_up}, we can show that \eqref{eq_equi_part_energy} holds also for $\Phi_{\Lambda}(\rho,0,U_{1,0},U_{2,0})$ so that $\pd{\Phi_{\Lambda}}{r}(\rho,0,U_{1,0},U_{2,0})\equiv 0$ implies that $U_{j,0}$ is $1$-homogeneous.
It follows that 
\begin{equation}
\Lambda_1|\Omega_{U_{1,0}}\cap B_1| + \Lambda_2|\Omega_{U_{2,0}}\cap B_1|
=\Phi_{\Lambda}(1,0,U_{1,0},U_{2,0})=\lim_{h \to \infty}\Phi_{\Lambda}(r_h,x,U_1,U_2).
\end{equation}
Hence, it is enough to show that for $j=1,2$
\begin{equation}\label{eq_measures_half}
|\Omega_{U_{j,0}}\cap B_1|=\omega_d/2
\end{equation}
to complete the proof. To this end, we note that, since $u_{1,0}$ is harmonic and  $\Omega_{U_{1,0}}$ and one-homogeneous 
\begin{equation}\label{eq_ui0_eigen}
-\Delta_{\mb S^{d-1}} u_{1,0}=(d-1)u_{1,0} \quad \text{ in } \Omega_{U_{1,0}} \cap \mb S^{d-1}.
\end{equation}
The Faber-Krahn inequality on $\mb S^{d-1}$ implies that $\mc{H}^{d-1}( \Omega_{U_{1,0}}\cap \mb S^{d-1}) \ge d \omega_d/2$. The one-homogeneity of $U_{1,0}$
then yields $|\Omega_{U_{1,0}}\cap B_1| \ge \omega_d/2$. Similarly, $|\Omega_{U_{2,0}}\cap B_1| \ge \omega_d/2$.
Since $|U_{1,0}||U_{2,0}|=0$, we have proved \eqref{eq_measures_half}.
\end{proof}

\begin{proposition}\label{prop_blo_up_flat}
Let  $U_{j,0} \in \mc{BU}_{U_j}(x)$ for $j=1,2$  for some blow up sequence $(U_j)_{r_h,x_h}$ with $\{x_h\} \subset \partial \Omega_{U_1} \cap \partial \Omega_{U_2}$.  
Then there exist $\nu \in \mathbb{S}^{d-1}$ and  $C_j \in \R^{k_j}$ with $|C_j|=\sqrt{\Lambda_j}$ for $j=1,2$ such that
\begin{equation}\label{eq_blow_up_explicit}
U_{1,0}(y)=C_1(y\cdot \nu)^+ \quad  \text{ and } \quad U_{2,0}(y)=C_2(y\cdot \nu)^-.
\end{equation}
\end{proposition}

\begin{proof}
We claim that for any $t>0$
\begin{equation}\label{proof_prop_blo_up_flat_1}
\Phi(t,x,U_{1,0},U_{2,0})=\frac{\omega_d}{2}(\Lambda_1 + \Lambda_2).
\end{equation}
To this end, by Proposition \ref{prop_mono_Weiss}, for any  $\e>0$ there exists $\rho>0$, depending only on $\e$, such that 
\begin{equation}
\frac{\omega_d}{2}(\Lambda_1 + \Lambda_2)\le \Phi_{\Lambda}(\rho,x,U_1,U_2) \le \frac{\omega_d}{2}(\Lambda_1 + \Lambda_2)+\e
\end{equation}
so that for any large enough $h \in \mathbb{N}$
\begin{equation}
\frac{\omega_d}{2}(\Lambda_1 + \Lambda_2)\le \Phi_{\Lambda}(\rho,x_h,U_1,U_2) \le \frac{\omega_d}{2}(\Lambda_1 + \Lambda_2)+2\e.
\end{equation}
By monotonicity, for any $t>0$ and $h$ large enough, that is, such that $t r_h \le \rho $,
\begin{equation}
\frac{\omega_d}{2}(\Lambda_1 + \Lambda_2)\le \Phi_{\Lambda}(t r_h,x_h,U_1,U_2) \le \frac{\omega_d}{2}(\Lambda_1 + \Lambda_2)+2\e.
\end{equation}
Equivalently 
\begin{equation}
\frac{\omega_d}{2}(\Lambda_1 + \Lambda_2)\le \Phi_{\Lambda}(t,0,(U_1)_{r_h,x_h},(U_2)_{r_h,x_h}) \le \frac{\omega_d}{2}(\Lambda_1 + \Lambda_2)+2\e
\end{equation}
thus, by Proposition \ref{prop_blow_up_limit_H1}, 
we can prove \eqref{proof_prop_blo_up_flat_1} passing to the limit as $h \to \infty$  and then as $\e \to 0^+$.

Arguing as in Proposition \ref{prop_mono_Weiss}, it follows that $U_{1,0}$ is one-homogeneous, that $|\Omega_{U_{j,0}}\cap B_1|=\omega_d/2$
and that $u_{i,0}$ solves \eqref{eq_ui0_eigen}.

By the Faber-Krahn inequality on $\mb S^{d-1}$ it follows that $\Omega_{U_{1,0}}\cap B_1=\{x \cdot \nu>0\}$ and consequently $\Omega_{U_{2,0}}\cap B_1=\{x \cdot \nu<0\}$ for some $\nu \in \mb{S}^{d-1}$.
Hence, \eqref{eq_ui0_eigen} implies that $u_{i,0}(y)=c_i (y \cdot \nu)^+$ for any $i \in A_1$ 
while $u_{i,0}(y)=c_i (y \cdot \nu)^-$ for any $i \in A_2$ for some $c_i \in \R$. Finally, by  \eqref{eq_shape_der_blow_up},  we conclude that $U_{j,0}$ must be as in \eqref{eq_blow_up_explicit}.
\end{proof}

\subsection{Viscosity solutions}\label{subsec_visco}
In this section we show that minimizers are solutions in a  viscosity sense to an overdetermined problem at two-phase and branching points.

\begin{definition}\label{def_touch}
Let $\Omega$ be an open set. 
We say that a continuous function $Q: \Omega \to \mathbb{R}$ touches a function $w: \Omega \to \mathbb{R}$ from below (resp. from above) at a point $x \in \Omega$ if
\begin{equation}
Q(x) = w(x) \quad \text{and} \quad Q(y) - w(y) \leq 0 \; (\text{resp. } Q(y) - w(y) \geq 0)
\end{equation}
for every $y$ in a neighborhood of $x$ and $Q$ is differentiable at $x$.
\end{definition}

Let $x_0$ be a two-phase or a branching point and $r \in (0,r_0)$. We are going to show that $U$ is a viscosity solution of the problem
\begin{equation}\label{prob_visc}
\begin{cases}
-\Delta U_j=0, \text{ in }B_{r}(x_0)\cap \Omega_{U_j} \text{ for } j=1,2,\\
|\nabla |U_1||^2-|\nabla |U_2||^2=\Lambda_1-\Lambda_2, \text{ on }B_r(x_0) \cap \partial\Omega_{U_1}\cap \partial\Omega_{U_2},\\
|\nabla |U_1||^2 \ge \Lambda_1,|\nabla |U_2||^2 \ge \Lambda_2 \text{ on }B_r(x_0) \cap \partial\Omega_{U_1}\cap \partial\Omega_{U_2},\\
|\nabla |U_1||^2=\Lambda_1\text{ on }B_r(x_0) \cap (\partial\Omega_{U_1}\setminus \partial\Omega_{U_2}),\\
|\nabla |U_2||^2=\Lambda_2 \text{ on }B_r(x_0) \cap (\partial\Omega_{U_2}\setminus \partial\Omega_{U_1}),\\
\end{cases}
\end{equation}
in the sense of Definition \ref{def_visc} below.

\begin{definition}\label{def_visc}
We say that $U$ is a solution of \eqref{prob_visc} if, letting $\tilde{U}=:|U_1|-|U_2|$, the following holds.
Suppose that $Q$ touches $\tilde{U}$ from below at $x$. Then:
\begin{itemize}
\item[(A.1)] If $x \in B_r(x_0) \cap (\partial \Omega_{U_1}\setminus \partial  \Omega_{U_2})$,
then $|\nabla Q^+(x)| \leq \sqrt{\Lambda_1}$.
\item[(A.2)] If $x \in B_r(x_0) \cap (\partial \Omega_{U_2}\setminus \partial  \Omega_{U_1})$, 
then $Q^+ \equiv 0$ in a neighborhood of $x$ and
\begin{equation}
|\nabla Q^-(x)| \geq \sqrt{\Lambda_2}  .
\end{equation}
\item[(A.3)] If $x \in B_r(x_0) \cap \partial \Omega_{U_1} \cap \partial  \Omega_{U_2}$,
then $|\nabla Q^-(x)| \geq \sqrt{\Lambda_2}$ and
\begin{equation}
|\nabla Q^+(x)|^2 - |\nabla Q^-(x)|^2 \le \Lambda_1-\Lambda_2.
\end{equation}
\end{itemize}
Suppose that $Q$ touches $\tilde{U}$ from above at $x$.  Then:
\begin{itemize}
\item[(B.1)] If $x \in B_r(x_0) \cap (\partial \Omega_{U_1}\setminus \partial  \Omega_{U_2})$, then $Q^- \equiv 0$ in a neighborhood of $x$ and
\begin{equation}
|\nabla Q^+(x)| \geq \sqrt{\Lambda_1}.
\end{equation}
\item[(B.2)] If $x \in B_r(x_0) \cap (\partial \Omega_{U_2}\setminus \partial  \Omega_{U_1}) $, 
then $|\nabla Q^-(x)| \leq \sqrt{\Lambda_2}$.
\item[(B.3)] If $x \in B_r(x_0) \cap \partial \Omega_{U_1} \cap \partial  \Omega_{U_2}$, 
then $|\nabla Q^+(x)| \geq \sqrt{\Lambda_1}$ and
\begin{equation}
|\nabla Q^+(x)|^2 - |\nabla Q^-(x)|^2 \ge \Lambda_1 - \Lambda_2.  
\end{equation}
\end{itemize}
\end{definition}

\begin{proposition}\label{prop_visc}
Let $x_0 \in  \partial \Omega_{U_1} \cap \partial \Omega_{U_2}$. Then  $U$ is a viscosity solution of  \eqref{prob_visc} in $B_r(x_0)$ for any $r \in (0,r_0)$ in the sense of Definition \ref{def_visc}.
\end{proposition}

\begin{proof}
It is well know, see \cite[Lemma 2.11]{MTV_reg_vect},  that at any $x \in (\partial \Omega_{U_1}\setminus \partial\Omega_{U_2}) 
\cup (\partial \Omega_{U_2}\setminus \partial\Omega_{U_1})$ any $U_{j,0} \in \mc{BU}_{U_J}(x)$ is of the form $U_{j,0}(y)=C_1(y\cdot \nu)^+$ with   $C_j \in \R^{k_j}$ and  $|C_j|=\sqrt{\Lambda_j}$ for $j=1,2$.
Indeed, $U$ minimizes the vectorial Bernoulli functional in a small ball $B_{r_1}(x_0)$.

Then, also taking Proposition \ref{prop_blo_up_flat} with $x_h=x$ into account for any point 
$x \in \partial \Omega_{U_1} \cap \partial\Omega_{U_2}$, we can argue as in \cite[Lemma 2.5]{DPSV_branch} and \cite[Lemma 5.2]{MTV_reg_spec} to conclude.
\end{proof}

\section{Regularity of two phase and branching points}\label{sec_reg}
We are going to follow the approach of \cite{MTV_reg_spec,MTV_reg_vect} to show that the free boundaries $\partial \Omega_{U_j}$ are smooth in a neighborhood  of two-phase and branching points. It is based on the notion of  Reifenberg flattens, which we now recall.

\begin{definition}[Reifenberg flattens]\label{def_Reif_flat}
Let $\Omega \subset \R^d$ be open, $\delta \in (0,1/2)$ and $R>0$. We say that $\Omega$ is a $(\delta,R)$-Reifenberg flat domain if: 
\begin{enumerate}[(i)]
\item For any $x \in \Omega$ and any $0\le r\le R$ there exists an hyperplane $H_{x,r}$ containing $x$ and such that 
\begin{equation}
d_{\mc{H}}(B_r(x) \cap H_{x,r},B_r(x) \cap \partial \Omega) \le r \delta,
\end{equation}
where $d_{\mc{H}}(\cdot, \cdot)$ denotes the distance in Hausdorff sense.
\item For any $x \in \Omega$ one of the connected components of $B_R(x) \cap \{y\in \R^d: d(y, H_{x,R})>2 \delta R\}$ is contained in $\Omega$ while the other one is contained in $\R^d\setminus \overline{\Omega}$ for any $x \in \partial \Omega$.
\end{enumerate}
\end{definition}

\begin{proposition}\label{prop_Reif_flat}
Let $x_0 \in  \partial \Omega_{U_1} \cap \partial \Omega_{U_2}$. Then for any $\delta\in (0,1/2)$ there exists $R>0$  and $r_1>0$ such that $\Omega_{U_j}$ is $(\delta,R)$-Reifenberg flat domain in $B_{r_1}(x_0)$ for $j=1,2$.
\end{proposition}

\begin{proof}
Let $\delta \in (0,1/2)$.  We argue by contradiction supposing  there  exist a sequence of points $\{x_h\}\subset \Omega_{U_1}\cup \Omega_{U_2}$ with 
$x_h \to x_0$ and a sequence $r_h \to 0^+$ such that for $U_j$ in $B_{r_h}(x_h)$ such that either (i) or (ii) do not hold in Definition \ref{def_Reif_flat}.
Let  $d_h:=d(x_h,  \partial \Omega_{U_1} \cap \partial \Omega_{U_2})$ and let $y_h \in \partial \Omega_{U_1} \cap \partial \Omega_{U_2}$ be such that
$d_h=|x_h-y_h|$. We distinguish two cases.

\textbf{Case 1:$|x_h-y_h|\le 2 r_h$.}  In this case, consider the blow up sequences $(U_j)_{y_h,r_h}$ for $j=1,2$ converging to $U_{1,0}=C_1(y\cdot \nu)^+$  and $U_{2,0}:=C_2(y\cdot \nu)^-$ some $\nu \in \mb{S}^{d-1}$ and  $C_j \in \R^{k_j}$ respectively thank to Proposition \ref{prop_blo_up_flat}. 
Letting $H_x:= x+\{y \in \R^d: y \cdot \nu=0\}$, by Proposition
\ref{prop_blow_up_limit_notzero_set_L1},  since $B_{r_h}(x_h) \subset B_{3r_h}(y_h)$,
\begin{multline}
d_{\mc{H}}(B_{r_h}(x_h) \cap (H_x-y_h),B_{r_h}(x_h) \cap \partial \Omega_{U_j}) \le d_{\mc{H}}(B_{3r_h}(y_h) \cap (H_x-y_h),B_{3r_h}(y_h) \cap \partial \Omega_{U_j})\\
= 3 r_h d_{\mc{H}}(B_1 \cap H_x,B_1 \cap \partial \Omega_{(U_j)_{3r_h,y_h}}) \le \delta r_h
\end{multline}
as soon as $r_h \le R$ with $R$ small enough. 

Furthermore, again by Proposition \ref{prop_blow_up_limit_notzero_set_L1}, letting $T_{R,\delta,x_h}:=B_R(x_h) \cap \{y\in \R^d: d(y, H_x)>2 \delta R\}$   and $T_{R,\delta,x}^{i}$ for $i=1,2$ its connected components, it follows that 
\begin{equation}
T_{R,\delta,x_h}^{i} \subset B_{R}(y_h) \cap  \Omega_{U_j} \subset  \Omega_{U_j} 
\end{equation}
for either $i=1$ or $i=2$, up to choosing a smaller $R>0$. Similarly 
\begin{equation}
T_{R,\delta,x_h}^{i} \subset B_{R}(y_h) \cap  (\R^d\setminus \overline{\Omega}_{U_j}) \subset \R^d\setminus \overline{\Omega}_{U_j}.
\end{equation}
Hence, we have reached a contradiction.

\textbf{Case 2:$|x_h-y_h| > 2 r_h$.}
We are going to reach a contradiction supposing for the sake of simplicity that $x_h \in \Omega_{U_1}$ for any $h \in \mathbb{N}$. This is not restrictive since we can reach a contradiction in the very same way if $x_h \in \partial \Omega_{U_2}$ for any $h \in \mathbb{N}$ while for the general case we can simply consider the subsequences of points $x_h$ in $ \partial \Omega_{U_1}$ and  $ \partial\Omega_{U_2}$. Furthermore, we may as well take $x_0=0$.

In view of Proposition \ref{prop_blo_up_flat} and Proposition \ref{prop_non_dege}, for any $\e>0$ there exists $r_1>0$ and $\nu \in \mathbb{S}^{d-1}$ such that 
$U_1$ is $\e$-flat in $B_{r_1}$ in the sense that 
\begin{equation}
|U_1(y)-C_1 (y\cdot\nu)^{+}| \le \e \quad \text{ and } \quad U_1\equiv 0 \text{ in } B_{r_1} \cap \{y \in \R^d: (y\cdot\nu)^{+}<-\e\}.
\end{equation}
Furthermore $(U_1)_{x_h,d_h}$ is a viscosity solution of 
\begin{equation}
\begin{cases}
-\Delta (U_1)_{x_h,d_h}=0, &\text{ in } \Omega_{(U_1)_{x_h,d_h}} \cap B_1,\\
|\nabla |(U_1)_{x_h,d_h}||=\sqrt{\Lambda_1}, &\text{ on } \partial\Omega_{(U_1)_{x_h,d_h}} \cap B_1,
\end{cases}
\end{equation}
in the  sense of \cite[Definition 1.2]{DT_impr} for any  $h \in \mathbb{N}$.  

Hence, choosing $\e>0$ small enough and thus  $r_1$ small enough, the improvement of flatness  \cite[Lemma 3.3]{DT_impr} holds   $(U_1)_{x_h,d_h}$. 
Furthermore, we may also suppose that  $\e < \delta$.

It follows that for any  there exists $R>0$, depending on $\delta$, such that for any $r \in (0,R)$
\begin{equation}
|(U_1)_{x_h,d_h}-C_1 (y\cdot\nu_{r,h})^{+}| \le \delta r \text{ in } B_r 
\quad (U_1)_{x_h,d_h}\equiv 0 \text{ in } B_r \cap \{y \in \R^d: y\cdot\nu_{r,h}<-2\delta r\},
\end{equation}
for some $\nu_{r,h} \in \mb{S}^{d-1}$ and $C_1 \in \R^{k_1}$ depending on $x_h$ and $r$.
In particular
\begin{align}
&B_r \cap \partial \Omega_{(U_1)_{x_h,d_h}}  \subset B_r \cap  \{y \in \R^d:-\delta r<y\cdot\nu_{r,h}<\delta r\},\\
&B_R \cap  \{y \in \R^d:y \cdot\nu_{r,h}>2\delta R\} \subset B_R \cap  \Omega_{(U_1)_{x_h,d_h}}, \\
&B_R\cap  \{y \in \R^d:y\cdot\nu_{r,h}<-2\delta R\} \subset B_R \cap  \R^d\setminus \overline{\Omega}_{(U_1)_{x_h,d_h}}.
\end{align}
It follows that, letting  $H_h:=\{y \in \R^d:y\cdot\nu_{r,h}=0\}$,
\begin{multline}
d_{\mc{H}}(B_{r_h}(x_h) \cap (H_h-x_h),B_{r_h}(x_h) \cap \partial \Omega_{U_1}) 
\le \frac{r_h}{d_h} d_{\mc{H}}(B_{d_h} \cap H_h,B_{d_h} \cap \partial \Omega_{(U_1)_{x_h,d_h}}) 
\le \delta r_h.
\end{multline}
Hence, we have reached a contradiction also in \textbf{Case 2}.
\end{proof}

From Proposition \ref{prop_Reif_flat}, it follows that $\Omega_{U_j}$ is a NTA domain in $B_{r_1}(x_0)$ in the sense of \cite[Theorem 3.4]{MTV_epsilon}, in view of \cite[Theorem 3.1]{KT_NTA} for some $r_1>0$ small enough.
Then by \cite[Theorem 5.2-7.9]{JK_boundary_har} and Remark \cite[Remark 3.5-3.8]{MTV_reg_vect} we have the following results.

\begin{theorem}\label{theor_conc_comp}
Let $x_0 \in  \partial \Omega_{U_1} \cap \partial \Omega_{U_2}$. There exists $r_1>0$ and $M>0$ such that 
for any $x \in B_{r_1}(x_0) \cap \partial\Omega_{U_j}$ and any  $r \in (0,r_1)$, there is exactly  one connected component of 
$B_r(x) \cap \Omega_{U_j}$ that intersects $B_{r/M}(x) \cap \Omega_{U_j}$.
\end{theorem}

\begin{theorem}[Boundary Harnack]\label{theor_part_harn}
Let $x_0 \in  \partial \Omega_{U_1} \cap \partial \Omega_{U_2}$. There exists $r_1>0$ such that for any $K \subset B_{r_1}(x_0)$
any positive harmonic function $u$ and any harmonic function $v$ both vanishing continuously on $B_{r_1}(x_0) \cap \partial \Omega_{U_j}$ we have for some constant $c>0$
\begin{equation}
c^{-1}\frac{v(y)}{u(y)}\le \frac{v(x)}{u(x)} \le c\frac{v(y)}{u(y)} 
\quad \text{ for any } x,y \in K \cap \partial \overline{\Omega}_{U_j}.
\end{equation}
Furthermore, there exists $\beta>0$ such that $v/u$ is $\beta$-H\"older continuous in $K \cap \partial \overline{\Omega}_{U_j}$. In particular, for any $y \in K \cap \partial \Omega_{U_j}$, there exists  the limit $\lim_{x \to y}\frac{v(x)}{u(x)}$.
\end{theorem}

In view of Proposition \ref{prop_blo_up_flat} and Theorems \ref{theor_conc_comp}-\ref{theor_part_harn}, we can then argue as in  \cite[Lemma 3.10]{MTV_reg_vect} to conclude that at least one component of $U_j$ must have a sign on $B_{r_1}(x_0) \cap \partial \Omega_{U_j}$. More precisely, we have the following.

\begin{lemma}\label{lemma_constant_sign}
Let $x_0 \in  \partial \Omega_{U_1} \cap \partial \Omega_{U_2}$. There exists $r>0$ and $i_j \in A_j$ such that $u_{i_j}$ has constant sign in 
$B_r(x_0)\cap \Omega_{U_j}$. Moreover, there is a constant $C>0$ such that $Cu_{i_j} \ge |U_j|$ in $B_r(x_0)\cap \Omega_{U_j}$.
\end{lemma}

We are finally in position to prove the key result of this section, that is, we are going to show that in some small ball $B_r(x_0)$ the positive components $u_{i_j}$ of $U_j$ given by Lemma \ref{lemma_constant_sign} are viscosity solutions of the problem 
\begin{equation}\label{prob_visc_gj}
\begin{cases}
-\Delta u_{i_j}=0, \text{ in }B_r(x_0)\cap \Omega_{U_j} \text{ for } j=1,2,\\
|\nabla u_{i_1}|^2-|\nabla u_{i_2}|^2=g_1\Lambda_1-g_2\Lambda_2, \text{ on }B_r(x_0) \cap \partial\Omega_{U_1}\cap \partial\Omega_{U_2},\\
|\nabla u_{i_1}|^2 \ge g_1\Lambda_1,|\nabla u_{i_2}|^2 \ge g_2\Lambda_2 \text{ on }B_r(x_0) \cap \partial\Omega_{U_1}\cap \partial\Omega_{U_2},\\
|\nabla u_{i_1}|^2=g_1\Lambda_1\text{ on }B_r(x_0) \cap (\partial\Omega_{U_1}\setminus \partial\Omega_{U_2}),\\
|\nabla u_{i_2}|^2=g_2\Lambda_2 \text{ on }B_r(x_0) \cap (\partial\Omega_{U_2}\setminus \partial\Omega_{U_1}),\\
\end{cases}
\end{equation}
for some H\"older continuous functions $g_j:B_r\cap \partial\Omega_{U_j} \to [c_0,1]$ with $c_0>0$.

We can give a definition of viscosity solution similar to \eqref{def_visc_gj} but with boundary data $g_j \Lambda_j$.
\begin{definition}\label{def_visc_gj}
We say that $(u_{i_1},u_{i_2})$ is a solution of \eqref{prob_visc} if, letting $u=:u_{i_1}-u_{i_2}$, the following holds.
Suppose that $Q$ touches $u$ from below at $x$. Then:
\begin{itemize}
\item[(A.1)] If $x \in B_r(x_0) \cap (\partial \Omega_{U_1}\setminus \partial  \Omega_{U_2})$,
then $|\nabla Q^+(x)| \leq \sqrt{g_1\Lambda_1}$.
\item[(A.2)] If $x \in B_r(x_0) \cap (\partial \Omega_{U_2}\setminus \partial  \Omega_{U_1})$, 
then $Q^+ \equiv 0$ in a neighborhood of $x$ and
\begin{equation}
|\nabla Q^-(x)| \geq \sqrt{g_2\Lambda_2}  .
\end{equation}
\item[(A.3)] If $x \in B_r(x_0) \cap \partial \Omega_{U_1} \cap \partial  \Omega_{U_2}$,
then $|\nabla Q^-(x)| \geq \sqrt{g_2\Lambda_2}$ and
\begin{equation}
|\nabla Q^+(x)|^2 - |\nabla Q^-(x)|^2 \le g_1\Lambda_1 - g_2\Lambda_2.
\end{equation}
\end{itemize}
Suppose that $Q$ touches $u$ from above at $x$.  Then:
\begin{itemize}
\item[(B.1)] If $x \in B_r(x_0) \cap (\partial \Omega_{U_1}\setminus \partial  \Omega_{U_2})$, then $Q^- \equiv 0$ in a neighborhood of $x$ and
\begin{equation}
|\nabla Q^+(x)| \geq \sqrt{g_1\Lambda_1}.
\end{equation}
\item[(B.2)] If $x \in B_r(x_0) \cap (\partial \Omega_{U_2}\setminus \partial  \Omega_{U_1}) $, 
then $|\nabla Q^-(x)| \leq \sqrt{g_2\Lambda_2}$.
\item[(B.3)] If $x \in B_r(x_0) \cap \partial \Omega_{U_1} \cap \partial  \Omega_{U_2}$, 
then $|\nabla Q^+(x)| \geq \sqrt{g_1\Lambda_1}$ and
\begin{equation}
|\nabla Q^+(x)|^2 - |\nabla Q^-(x)|^2 \geq g_1\Lambda_1 - g_2\Lambda_2.  
\end{equation}
\end{itemize}
\end{definition}

\begin{proposition}\label{prop_visc_gj}
Let $x_0 \in  \partial \Omega_{U_1} \cap \partial \Omega_{U_2}$. Let $r>0$ such that $u_{i_j}>0$ in $B_r(x_0)\cap \Omega_{U_j}$. 
Then there exists a constant $c_0>0$ and H\"older continuous functions $g_j:B_r\cap \partial\Omega_{U_j} \to [c_0,1]$ such that $(u_{i_1},u_{i_2})$ is a viscosity solution of \eqref{prob_visc} in the sense of Definition \ref{def_visc}.
\end{proposition}

\begin{proof}
In view of Proposition \ref{prop_visc}, we can follow the proof of \cite[Lemma 3.11]{MTV_reg_vect}. 
Let $r_1>0$ be such that in $B_{2r_1}(x_0)$ both Theorem \ref{theor_part_harn} and Lemma \ref{lemma_constant_sign} holds with $Cu_{i_j} \ge |U_j|$ for some $i_j \in A_j$. In particular, for any $x \in B_{r_1}(x_0)\cap \partial \Omega_{U_j}$ we may define the $\beta$-H\"older continuous function 
\begin{equation}
g_i: B_{r_1}(x_0) \cap \partial \Omega_{U_j} \to \R, \quad g_i(x):=\lim_{y \to x, y \in \Omega_{u_j}}u_i(y)/u_{i_j}(y).
\end{equation}
We may extend the definition of $g_i$ to $B_{r_1}(x_0) \cap \Omega_{U_j}$ as $g_i:=u_i/u_{i_j}$, which, by Theorem \ref{theor_part_harn}, is still  $\beta$-H\"older continuous. Furthermore,  it follows that
\begin{equation}
u_i= g_iu_{i_j} \text{ on } B_{r_1}(x_0) \cap \overline{\Omega}_{U_j},
\quad \text{ and } \quad  u_{i_j}= g |U_j| \text{ on } B_{r_1}(x_0) \cap \overline{\Omega}_{U_j},
\end{equation}
where $g:=(1+\sum_{i\in A_j,i\neq i_j} u_i^2)^{-1/2}$. By definition, $g:B_{r_1}(x_0)\cap \partial \Omega_{U_j} \to \R$ is $\beta$-H\"older continuous and takes values in $[C^{-1},1]$ since $Cu_{i_j} \ge |U_j|$.

Hence, we can proceed as in \cite[Lemma 3.11]{MTV_reg_vect} to deduce that  $(u_{i_1},u_{i_2})$ is a viscosity solution of \eqref{prob_visc_gj} from  
Proposition \ref{prop_visc}.
\end{proof}

By Proposition \ref{prop_visc_gj}, Theorem \ref{theo_reg_diri}  follows from \cite[Theorem 1.1]{DPSV_branch}. To be more precise, we notice that \cite[Theorem 1.1]{DPSV_branch} only deals with constant weights on $\partial \Omega_{U_1}, \partial \Omega_{U_2}$ but it can be adapted to the case of strictly positive and H\"older continuous weights, as it was done for instance in \cite{D_impr} for the one-phase problem.

{\bf Acknowledgments.}
G.Siclari. is supported by the GNAMPA project E53C25002010001  \emph{Asymptotic analysis of variational problems}. B. Velichkov was supported by the European Research Council's (ERC) project n.853404 ERC VaReg - \it Variational approach to the regularity of the free boundaries \rm, financed by the program Horizon 2020. 
B. Velichkov also acknowledges the MIUR Excellence Department Project awarded to the Department of Mathematics (CUP I57G22000700001) and also the support from the project MUR-PRIN “NO3” (n.2022R537CS).

\bibliographystyle{acm}
\bibliography{references}	
\end{document}